\documentclass{amsart}
\usepackage{amsmath,amsthm,amssymb}
\usepackage{tikz-cd}
\usepackage{thmtools}
\usepackage{enumitem}
\usepackage[hidelinks]{hyperref}
\usepackage{epigraph}
\usepackage{quiver}
\usepackage{mathrsfs}
\usepackage[noabbrev,capitalise]{cleveref}
\crefformat{equation}{(#2#1#3)}
\crefformat{enumi}{#2#1#3}
\crefformat{enumii}{#2#1#3}

\usepackage{enumitem}
\setlist{listparindent = \parindent, parsep=0pt,}
\setenumerate[1]{label = (\roman*), ref = (\roman*)}
\setenumerate[2]{label = (\alph*), ref = (\alph*)}
\title{On the theories classified by an \'etendue}
\author{J. L. Wrigley}
\address[Joshua L. Wrigley]{Université Paris Cité, CNRS, IRIF, F-75013, Paris, France.}
\email{wrigley@irif.fr}
\urladdr{https://jlwrigley.github.io/}

\theoremstyle{plain}
\newtheorem*{mainthm}{Theorem}
\newtheorem{thm}{Theorem}[section]
\newtheorem{lem}[thm]{Lemma}
\newtheorem{coro}[thm]{Corollary}
\newtheorem{prop}[thm]{Proposition}

\theoremstyle{definition}
\newtheorem*{introdf}{Definition}
\newtheorem{df}[thm]{Definition}

\newtheorem{rem}[thm]{Remark}

\newtheorem{ex}[thm]{Example}

\usepackage[
backend=biber,
sorting=nyt,
style=ieee
]{biblatex}
\renewcommand{\phi}{\varphi}

\newcommand{\id}{\text{id}}
\newcommand{\N}{\mathbb{N}}

\newcommand{\theory}{\mathbb{T}}

\newcommand{\Dtheory}{\mathbb{D}}

\newcommand{\Set}{{\bf Set}}

\newcommand{\topos}{\mathcal{E}}
\newcommand{\ftopos}{\mathcal{F}}
\newcommand{\Sub}{\mathrm{Sub}}
\newcommand{\Sh}{\mathbf{Sh}}
\newcommand{\X}{\mathbb{X}}
\newcommand{\cat}{\mathcal{C}}
\newcommand{\Topos}{\mathbf{Topos}}
\newcommand{\Frm}{\mathbf{Frm}}

\newcommand{\lrset}[2]{\left\{{#1}\,\middle\vert\,{#2}\right\}}
\newcommand{\sett}[2]{\{{#1}\,\vert\,{#2}\}}

\newcommand{\type}{\mathfrak{t}}
\newcommand{\power}{\mathscr{P}}

\newcommand{\Vectf}{\mathbb{V}_F}
\newcommand{\Vectff}{\mathbb{V}_F^f}
\newcommand{\0}{{\bf 0}}

\newcommand{\stab}{\mathrm{stab}}

\usepackage[a4paper,
left=1.5in,
right=1.5in,
top=1.1in,
bottom=1.1in
]{geometry}

\begin{document}
	
	\begin{abstract}
		We give a model-theoretic characterisation of the geometric theories classified by \'etendues -- the `locally localic' topoi.  They are the theories where each model is determined, syntactically and semantically, by any witness of a fixed collection of formulae.
		
		\smallskip
		\noindent \textsc{\keywordsname.}{
			\'Etendue, co-ordinatisable theory, uniform rigidity.
		}
		
		\smallskip
		\noindent \textsc{\subjclassname.}{
			03G30, 18F10, 18B25, 03C50.
		}
	\end{abstract}

	\maketitle

	\section*{Introduction}
	In this paper, we establish a logical characterisation of a class of Grothendieck topoi called \emph{\'etendues}.  Every topos classifies a theory of \emph{geometric logic} (\cite[\S D3.1]{elephant},\cite[\S 2.1]{TST}), where we recall that the classifying topos of a theory $\theory$ is a topos $\topos_\theory$ such that morphisms $\ftopos \to \topos_\theory$ correspond to models of $\theory$ internal to $\ftopos$, or more precisely there is an equivalence of pseudofunctors ${\bf Topos}(-,\topos_\theory) \simeq \theory\text{-}{\bf Mod}(-)$ (the morphisms of topoi we consider are \emph{geometric morphisms}, \cite[Definition A4.1.1]{elephant}).  As the name suggests, classifying topoi are analogous to classifying spaces from algebraic topology (cf.\ \cite[\S IV.2]{SGA4-3} and \cite{milnor}).  Specific classes of topoi can be carved out in terms of which geometric theories they classify.  The topoi of sheaves on a locale, for example, can be characterised logically as the classifying topoi for \emph{propositional} geometric theories (\cite[Remark D3.1.14]{elephant}).
	
	Not far beyond the topoi of sheaves on a locale, we encounter \emph{\'etendues}.  As expressed in \cite[\S 3]{lawvere_etendue}, an {\'etendue} is a topos that is `locally sheaves on a locale', or explicitly:
	\begin{introdf}[\S IV.9.8.2(e) \cite{SGA4-3}]
		A topos $\topos$ is an \emph{\'etendue} if there exists an object $A \in \topos$ whose support $A \to 1$ is an epimorphism (we say that $A$ is \emph{inhabited}) and whose slice $\topos/A$ is equivalent to a topos of sheaves on a locale.
	\end{introdf}
	Through their interplay with \emph{\'etale topological groupoids} (cf.\ \cref{ex:etale_grpd}), \'etendues have been promoted as a topos-theoretic framework in which to study various \emph{non-commutative} generalisations of Stone duality, where a dual equivalence is established between a class of {inverse semigroups} and a class of {\'etale topological groupoids} \cite{lawson_boolean_monoids,lawson_semigroups,resende,cockettgarner}.  In addition, an important body of literature has coalesced around \'etendues concerning their relation to \emph{foliations} and \emph{orbifolds} from differential geometry \cite{local_eqv,moer_fourier,moerpronk_orbifolds,rosenthal_local_eqv}.  It is not within the scope of the current paper to explore these applications, and instead the evocative accounts in \cite{hofstra_funk_invitation}, for  the connection to semigroup theory, and in \cite{moer_zara}, for the connection to foliation theory, are recommended.
	
	Our characterisation of the theories classified by an \'etendue falls into the well-worn logical paradigm of a syntactic characterisation of semantic phenomena\footnote{A characterisation for \emph{atomic} topoi with a similarly model-theoretic flavour is presented in \cite{caramello_atomic}, giving a counterpart to the Ryll-Nardzewski theorem \cite[Theorem 7.3.1]{hodges} for geometric logic.}.  For an example of such a interplay of syntax and semantics, consider the following situation: let $\theory$ be a (classical) first-order theory and $\phi$ a unary formula in the language, i.e.\ with one free variable.  Gaifman's co-ordinatisation theorem asserts that the following are equivalent:
	\begin{enumerate}
		\item given a pair of models $M,N \vDash \theory$, each isomorphism of the induced substructures $\phi(M) \cong \phi(N)$ extends to a unique isomorphism $M \cong N$ (we say that $\theory$ is \emph{rigidly, relatively categorical} over $\phi$),
		\item for any model $M \vDash \theory$, every element of $M$ is in the \emph{definable closure} of $\phi(M)$ (see \cite[Exercise 1.4.10]{marker} or \cref{df:coord} for the definition of definable closure; we say that $\theory$ is \emph{co-ordinatised} over $\phi$).
	\end{enumerate}
	(See \cite[\S 7]{relcat} for a proof; as shown in \cite{evanshewitt}, the uniqueness assumption on the extended isomorphism, i.e.\ rigidity, cannot be removed from the above equivalence.)

We will logically characterise \'etendues by considering similar conditions.  We delay the precise statement of our characterisation theorem (\cref{thm:mainthm_geometric}) until after the relevant terminology has been developed, but if we restrict to the classifying topoi of theories of classical logic we obtain the following simplification that closely resembles Gaifman's co-ordinatisation theorem:
	\begin{mainthm}\label{maintheorem}
		For a classical theory $\theory$ classified by a topos $\topos$, the following are equivalent:
		\begin{enumerate}[label = (\Alph*)]
			\item\label{enum:coord} $\theory$ is uniformly co-ordinatisable,
			\item\label{enum:relcat} $\theory$ is {uniformly rigid} (in sets),
			\item\label{enum:etendue} $\topos$ is an {\'etendue}.
		\end{enumerate}
	\end{mainthm}
Essentially, the prefix `uniformly' means that, rather than each model $M \vDash \theory$ being determined by the substructure $\phi(M)$, the model $M$ is determined by \emph{any element} of $\phi(M)$.  A prototypical example of such a theory, that we discuss further in \cref{ex:vectff}, is given by vector spaces: any $n$-dimensional vector space (over a finite field) is determined up to isomorphism by a basis, i.e.\ any witness for the formula $B_n(\vec{x})$ expressing that $\vec{x}$ is an $n$-tuple of maximally linearly independent vectors.  We should not be surprised at the appearance of a notion of `co-ordinatisation' in the characterisation of \'etendues, given the informal remark of Hofstra and Funk in \cite[\S 5.4]{hofstra_funk_invitation}.

	
	\subsection*{Overview} 
	The ideas behind the proof of our characterisation can be summarised as follows: we use two characterisations of localic topoi.  Firstly, a topos is localic if and only if every object admits a covering by subterminal objects (\cite[Theorem 5.37]{topos}).  Translating this condition to the setting of the slice of a classifying topos yields the notion of uniform co-ordinatisation and the equivalence \cref{enum:coord} $\iff$ \cref{enum:etendue} (see \cref{prop:prebound_cov_implies_localic} and \cref{coro:etendue_implies_coord_over_set}).  Next, a topos is localic if and only if the canonical morphism to its \emph{localic reflection} is an equivalence (\cite[\S A4.6]{elephant}), which occurs if and only if post-composing with this morphism yields an equivalence on hom-categories.  Faithfulness of post-composition is expressed by the notion of (strong) uniform rigidity (fullness and essential surjectivity is automatic for a classical theory, see \cref{sec:coherent}), and thus we obtain the equivalence \cref{enum:relcat} $\iff$ \cref{enum:etendue} (see \cref{coro:mainthm_classical}).
	
	The paper proceeds as follows.  In \cref{sec:coord}, we introduce the notion of uniform co-ordinatisability (\cref{df:coord}) and prove that the classifying topos of a geometric theory is an \'etendue if and only if the theory is uniformly co-ordinatised (\cref{coro:etendue_implies_coord_over_set}).  
	
	In \cref{sec:rigid}, we define \emph{strong} uniform rigidity and uniform rigidity (\cref{df:rigid}).  We also discuss the notion of \emph{realising} (model-theoretic) types in the arbitrary topos-theoretic context (\cref{df:realise-types}).  We show that the classifying topos of a geometric theory is an \'etendue if and only if the theory is strongly uniformly rigid (in any topos) and satisfies a condition concerning realisation of types (\cref{thm:mainthm_geometric}).  If the theory is coherent or classical, any mention of realising types can be dropped (\cref{sec:coherent}).
	
	Finally, in \cref{sec:appl}, we discuss the interplay between our logical characterisation of \'etendues and some previously known facts, namely that an \'etendue has a small set of points and that every topos is a hyperconnected quotient of an \'etendue.
	
	\subsection*{Prerequisites and notation}  By topos, we mean a Grothendieck topos.  We will assume knowledge of geometric logic and classifying topos theory (textbook accounts can be found in \cite[Part D]{elephant}, \cite[\S 2]{TST} or \cite[\S 8]{MR}).  We phrase the syntax of geometric logic as Gentzen-style sequents but leave the free variables implicit by context, writing $\phi(\vec{x},\vec{y}) \vdash \psi(\vec{x},\vec{z})$ to mean ``for all $\vec{x},\vec{y},\vec{z}$, if $\phi(\vec{x},\vec{y})$ then $\psi(\vec{x},\vec{z})$''.  We write $A \subseteq B$ to say that $A$ is a \emph{subobject} of $B$.  For convenience, we will assume that our geometric theories are single-sorted, which can be ensured by \cite[Lemma D1.4.13]{elephant}, but our argument is easily modified to the multi-sorted context.  
	
	Eventually, beginning in \cref{rem:model_in_any_topos}, we will discuss models of a first-order theory internal to any topos (see \cite[\S D1.2]{elephant} for how to interpret first-order logic inside an arbitrary topos), but initially all models are assumed to be \emph{set-based}.

	\section{\'Etendues and uniform co-ordinatisability}\label{sec:coord}
	In this first section, we introduce the notion of a uniformly co-ordinatisable theory and show that a geometric theory is uniformly co-ordinatisable if and only if its classifying topos is an \'etendue.  Recall that a topos $\topos$ is an \'etendue if there exists an inhabited object $A \in \topos$ for which the slice topos $\topos/A$ is localic.  We say that $\topos$ is an \'etendue \emph{over} $A$.  Recall also that the slice topos $\topos/A$ comes equipped with an \emph{\'etale} geometric morphism $\pi_A \colon \topos/A \to \topos$ whose inverse image $\pi_A^\ast$ sends an object $B \in \topos$ to the product projection $B \times A \to A$, considered as an object of $\topos/A$ (\'etale geometric morphisms are also called \emph{locally homeomorphic} geometric morphisms, see \cite[p.~651]{elephant}, since an \'etale geometric morphism between topoi of sheaves on a space corresponds to a local homeomorphism of spaces).
	\begin{ex}[$G$-sets]
		The topos $\Set^G$ of $G$-sets, for a group $G$, is a primordial example of an \'etendue.  Indeed, the slice $\Set^G/G$, over $G$ equipped with its canonical $G$-action, is equivalent to $\Set$.
	\end{ex}
	\begin{ex}[Sheaves on an \'etale groupoid]\label{ex:etale_grpd}
		The following example encompasses the former and serves to describe all possible \'etendues.  A localic groupoid $\X = (X_1 \rightrightarrows X_0)$ is \emph{\'etale} when the source and target maps $s, t \colon X_1 \rightrightarrows X_0$ are both \'etale maps (i.e.\ local homeomorphisms).  The topos $\Sh(\X)$ of equivariant sheaves on $\X$ is an \'etendue.  The source map $ X_1 \to X_0$ can be equipped with an $X_1$-action and so defines an object of $\Sh(\X)$ for which $\Sh(\X){/}X_1\simeq \Sh(X_0)$.  
		
		Moreover, every \'etendue $\topos$ is of the form $\Sh(\X)$ for some \'etale localic groupoid $\X$.  This easily follows by an application of Joyal-Tierney descent for topoi \cite{JT}.  If $\topos$ is an \'etendue over $A \in \topos$, then the \'etale morphism $\Sh(X_0) \simeq \topos/A \to \topos$ is an effective descent morphism of topoi (being open and surjective) with localic domain, and so the topos $\topos$ is equivalent to the topos of sheaves on the localic groupoid whose locale of objects is the locale $X_0$ and whose locale of arrows and source and target maps are obtained via the (bi-)pullback of topoi
		\[
		\begin{tikzcd}
			\Sh(X_1) &[-30pt] \simeq \Sh(X_0) \times_\topos \Sh(X_0) \ar{rr}{s} \ar{d}[']{t}  \ar["\lrcorner"{anchor=center, pos=0.125}, draw=none]{rrrd} &[-60pt] & \Sh(X_0) \simeq &[-30pt] \topos{/}A \ar{d}{\pi_A} \\
			& \Sh(X_0) & \simeq \topos{/}A \ar{rr}[']{\pi_A} && \topos.
		\end{tikzcd}
		\]
		The morphisms $s$ and $t$ are \'etale since \'etale geometric morphisms are preserved by pullback (which follows from \cite[Corollary 4.35]{topos}).  See \cite[\S 3]{mono_site} for a more in-depth discussion.  If $\topos$ is a topos \emph{with enough points} (\cite[Defintion 1.1.22]{TST}), then we can instead use \'etale \emph{topological} groupoids (\cite[\S IV.9.8.2]{SGA4-3}).
	\end{ex}
	\subsection{Co-ordinatisable theories}  We now recall the notion of \emph{co-ordinatisable} theory from \cite[\S 7]{relcat} and introduce our `uniform' specialisation.
	\begin{ex}\label{ex:vectff}
		It is useful to have a basic example to revisit when each new notion is introduced: we will use finite dimensional vector spaces over a finite field $F$.  Let $\Vectf$ denote the classical theory of $F$-vector spaces.  We can view $\Vectf$ as a coherent theory (also called a positive theory, i.e.\ avoiding mention of the negation symbol) via a process known as \emph{Morleyization} -- for each positive formula $\phi(\vec{x})$, we add a formal relation symbol $(\neg \phi)(\vec{x})$ and axioms expressing that $(\neg \phi)(\vec{x})$ is the negation of $\phi(\vec{x})$ (see \cite[Lemma D1.5.13]{elephant}).  As far as geometric logic is concerned, this means we can negate any formula that does not include an infinitary disjunction.  
		
		We write $B_n(x_1,\dots,x_n)$ as shorthand for the formula
		\begin{multline*}\textstyle
			\bigwedge_{\vec{\lambda} \in F^n \setminus \vec{0}} (\lambda_1 x_1 + \dots + \lambda_n x_n \neq \0 )
			\textstyle \land  \,  \neg 
			\exists y. \ 
			\bigwedge_{\vec{\lambda} \in F^{n+1} \setminus \vec{0}} (\lambda_1 x_1 + \dots + \lambda_n x_n + \lambda_{n+1}y \neq \0)  
		\end{multline*}
		expressing that $(x_1, \dots , x_n)$ is a maximally linearly independent set of vectors, i.e.\ a basis of length $n$.  The geometric theory of \emph{finite dimensional $F$-vector spaces} $\Vectff$ is given by adding to $\Vectf$ the geometric axiom $\top \vdash \bigvee_{n \in \N} \exists \vec{x}. \ B_n(\vec{x})$.
		
		Let $V \vDash \Vectff$ be a model that satisfies $V \vDash \exists \vec{x}. \ B_n(\vec{x})$, i.e.\ $V$ is an $n$-dimensional vector space.  It goes without saying that any witness $(v_1, \dots , v_n) \in V^n$ of $B_n(\vec{x})$ determines $V$ as any other vector $v' \in V$ can be expressed as a linear combination of $(v_1, \dots , v_n)$. 
		The notion of \emph{uniform co-ordinatisability} intends to capture this phenomenon.  It is a strengthening of the co-ordinatisability property discussed in \cite{relcat}.
	\end{ex}
	Let $\theory$ be a geometric theory and let $\Psi$ be a set of geometric formulae in the language of $\theory$ (note that we allow the number of free variables of each $\psi(\vec{x}) \in \Psi$ to vary).  Given a model $M \vDash \theory$, we write $\psi(M)$ for the interpretation of $\psi(\vec{x})$ in $M$, i.e.\ the set $\lrset{\vec{m} \in M^n}{M \vDash \psi(\vec{m})}$, and we write $\Psi(M)$ for the coproduct $\coprod_{\psi \in \Psi} \psi(M)$.
	\begin{df}\label{df:coord}
		\begin{enumerate}
			\item (\cite[\S 7]{relcat}) The theory $\theory$ is \emph{co-ordinatised} over $\Psi$ if, for any model $M \vDash \theory$, every $m \in M$ is in the definable closure of $\Psi(M)$, i.e.\ for every $m \in M$ there exists a geometric formula $\phi(\vec{x},y)$ and $\vec{n} \in \Psi(M)^{|\vec{n}|}$ for which $m$ is the unique element of $M$ such that $M \vDash \phi(\vec{n},m)$. 
			
			\item We say that $\theory$ is \emph{uniformly co-ordinatised} over $\Psi$ if, for each $\psi(\vec{x}) \in \Psi$, there is a set of geometric formulae $\Theta_\psi$ of the form $\theta(\vec{x},y)$, such that $\theory$ proves the sequents \begin{align*}
				\theta(\vec{x},y) & \vdash \psi(\vec{x}) , \\
				\theta(\vec{x},y) \land \theta(\vec{x},y') & \vdash (y = y'), \\
				\psi(\vec{x}) \land (y=y) & \vdash \textstyle \bigvee_{\theta \in \Theta_\psi}  \theta(\vec{x},y).
			\end{align*}
		\end{enumerate}
		We will say that $\theory$ is \emph{uniformly co-ordinatisable} if it is uniformly co-ordinatised over some \emph{inhabited} $\Psi$, i.e.\ a set of formulae for which $\theory$ proves the sequent $\top \vdash \bigvee_{\psi \in \Psi} \exists \vec{x}. \ \psi(\vec{x})$.
	\end{df}
	Note that, for each model $M \vDash \theory$ of a {uniformly co-ordinatised} theory, every $m \in M$ is in the definable closure of the singleton $\{\vec{n}\}$, for any $\vec{n} \in \Psi(M)$.  Hence, if $\theory$ is uniformly co-ordinatised over $\Psi$, it is co-ordinatised over $\Psi$.  The adjective \emph{uniform} is intended to convey the idea that the choice of parameters $\vec{n} \in \Psi(M)$ must be chosen `uniformly' when co-ordinatising the elements of $M$ (cf.\ \cref{rem:switch_quantifiers}).  The parameters $\vec{n} \in \Psi(M)$ behave as a `logical basis' for the model $M$, analogous to the notion of basis in linear algebra, while $\Psi$ plays the role of a definable set of `logical bases'.  Trivially, any theory is co-ordinatised over the formula $(x=x)$, and so clearly there are plenty of examples of theories that are co-ordinatisable but not uniformly co-ordinatisable.
	\begin{ex}
		The theory $\Vectff$ is uniformly co-ordinatised over $\lrset{B_n(\vec{x})}{n \in \N}$, the witnessing formulae being given by $\sett{B_n(\vec{x}) \land (\lambda_1 x_1 + \dots + \lambda_n x_n = y)}{\vec{\lambda} \in F^n}$.
	\end{ex}
	\begin{rem}\label{rem:theta_non-empty}
		Suppose that, for some $\psi' \in \Psi$, the set $\Theta_{\psi'}$ is empty.  Then $\theory$ proves the sequent $\psi'(\vec{x}) \land (y=y) \vdash \bigvee \emptyset \equiv \bot$ and so $\psi'(\vec{x})$ is $\theory$-provably equivalent to $\bot$, i.e.\ $\psi'(\vec{x})$ is inconsistent with $\theory$.  Therefore, $\theory$ still proves the sequent $\top \vdash  \bigvee_{\psi \in \Psi\setminus\{\psi'\}} \exists \vec{x}. \ \psi(\vec{x})$ and thus we can safely assume that, for all $\psi \in \Psi$, the set $\Theta_\psi$ is non-empty.
	\end{rem}
	\begin{rem}
		Suppose $\theory$ is uniformly co-ordinatised over a set of geometric formulae $\Psi$, then the set of co-ordinatising formulae $\Theta_\psi$, for each $\psi(\vec{x}) \in \Psi$, can be taken to consist only of \emph{regular} formulae, i.e.\ formulae built only using the logical connectives $\{\,=,\exists,\land\,\}$ (\cite[Definition D1.1.3]{elephant}).  Given a geometric formula $\theta(\vec{x},y) \in \Theta_\psi$, recall from \cite[Lemma D1.3.8]{elephant} that, up to provable equivalence, $\theta(\vec{x},y) $ can be rewritten as an infinite disjunction $\bigvee_{i \in I_\theta} \theta_i(\vec{x},y)$ where each $\theta_i(\vec{x},y)$ is a {regular} formula.  Note that, for each $i \in I_\theta$, the theory $\theory$ still proves the following sequents:
		\begin{align*}
			\theta_i(\vec{x},y) \vdash \theta(\vec{x},y) & \vdash \psi(\vec{x}), \\
			\theta_i(\vec{x},y) \land \theta_i(\vec{x},y') \vdash \theta(\vec{x},y) \land \theta(\vec{x},y') & \vdash (y=y'), \\
			\psi(\vec{x}) \land (y=y) & \textstyle \vdash \bigvee_{\theta \in \Theta_\psi} \theta(\vec{x},y) \equiv \bigvee_{\theta \in \Theta_\psi} \bigvee_{i \in I_\theta} \theta_i(\vec{x},y),
		\end{align*}
		and so the set of geometric formulae $\Theta_\psi$ can be replaced by the set of regular formulae $\bigcup_{\theta \in \Theta_\psi} \lrset{\theta_i(\vec{x},y)}{i \in I_\theta}$.
	\end{rem}
	\begin{rem}
		The definition of a theory co-ordinatised over a formula $\Psi$ given in \cref{df:coord} differs from the notion presented in \cite{relcat} in that \cref{df:coord} refers to geometric logic, whereas \cite{relcat} exists within classical first-order logic.  We briefly confirm that the definable closure of $\Psi(M)$ using geometric formulae coincides with the definable closure of $\Psi(M)$ using coherent formulae.  Given a model $M \vDash \theory$, suppose that $m \in M$ is the unique element such that $M \vDash \phi(\vec{n},m)$, for some $\vec{n} \in \Psi(M)^{|\vec{n}|}$ and geometric formula $\phi(\vec{x},y)$.  By \cite[Lemma D1.3.8]{elephant}, $\phi(\vec{x},y)$ is provably equivalent to an infinite disjunction of positive formulae $\bigvee_{i \in I} \phi_i(\vec{x},y)$, and so $M \vDash \phi_i(\vec{n},m)$ for some $i \in I$.  Given another element $m' \in M$ for which $M \vDash \phi_i(\vec{n},m')$, then clearly $M \vDash \phi(\vec{n},m')$ too.  Thus, we conclude that $m = m'$. 
	\end{rem}
	\begin{rem}\label{rem:switch_quantifiers}
		Informally, uniform co-ordinitisation can be seen as a variant of co-ordinatisation where the quantification of clauses has been changed, in the sense that a theory $\theory$ is co-ordinatised over, for simplicity, a single unary formula $\psi(x)$ if and only if: 
		\begin{multline*}
			\forall M \vDash \theory \text{ a model}. \ 
			\exists \, \Phi_M  \text{ a set of geometric formulae}.  \
			\forall m \in M. \\ 
			\exists \vec{n} \in \psi(M)^{|\vec{n}|}. \
			\exists \phi(\vec{x},y) \in \Phi_M. \
			M \vDash ( \phi(\vec{n},m) \land \forall y. \ ( \phi(\vec{n},y) \rightarrow y = m )),
		\end{multline*}
		whereas $\theory$ is uniformly co-ordinatised over $\psi(x)$ only if: 
		\begin{multline*}
			\exists \, \Theta \text{ a set of binary geometric formulae}. \ 
			\forall M \vDash \theory \text{ a model}. \ 
			\forall n \in \psi(M). \\
			\forall m \in M. \ 
			\exists \theta \in \Theta. \ 
			M \vDash ( \theta(n,m) \land \forall y. \ ( \theta(n,y) \rightarrow y = m )).
		\end{multline*}
		The above `only if' becomes an equivalence if we can equate provability modulo $\theory$ with satisfiability in set-based models, i.e.\ we assume that the geometric theory $\theory$ has \emph{enough set-based models} (\cite[Definition 1.4.17]{TST}).  (This is normally known as the `completeness theorem', but for infinitary logics such as geometric logic, the completeness theorem fails, see for instance \cite[Example C1.2.8]{elephant}.)
	\end{rem}
	\begin{rem}[cf.\ Proposition 7.2 \cite{relcat}]\label{rem:finite_models}
		We are restricted to \emph{finite} model theory if we only consider coherently uniformly co-ordinatised theories, that is a coherent (i.e.\ positive) theory $\theory$ that is uniformly co-ordinatised over an inhabited set of coherent formulae $\Psi$.  Since $\theory$ proves the sequent $\top \vdash \bigvee_{\psi \in \Psi} \exists \vec{x}. \ \psi(\vec{x})$, by the compactness theorem for coherent logic we have that $\theory$ proves $\top \vdash \bigvee_{\psi \in \Psi'} \exists \vec{x}. \ \psi(\vec{x})$ for some finite subset $\Psi' \subseteq \Psi$.  Similarly, as $\theory$ proves $\psi(\vec{x}) \land (y=y) \vdash \bigvee_{\theta \in \Theta_\psi} \theta(\vec{x},y)$, for each $\psi \in \Psi$, a second application of compactness gives us that $\theory$ proves $\psi(\vec{x}) \land (y=y) \vdash \bigvee_{\theta \in \Theta'_\psi} \theta(\vec{x},y)$ for some finite subset $\Theta'_\psi \subseteq \Theta_\psi$.  Thus, for any set-based model $M$ of $\theory$, there must be a witness $\vec{m} \in M$ for one of the finitely many formulae in $\Psi'$, and once such a witness of $\psi \in \Psi'$ is fixed, every other element $n \in M$ is the unique element satisfying $\theta(\vec{m},y)$ for one of the finitely many formulae $\theta \in \Theta'_\psi$.  Hence, every set-based model of $\theory$ is finite (cf.\ also \cref{rem:small_points_via_logic}).
	\end{rem}
	\subsection{\'Etendues are unifromly co-ordinatisable}
	We now prove that a geometric theory is classified by an \'etendue if and only if it is uniformly co-ordinatisable.  Recall that the classifying topos $\topos_\theory$ of a geometric theory $\theory$ contains the \emph{syntactic category} of $\theory$ as a full subcategory, and moreover this subcategory generates the topos (see, for instance, \cite[\S 8]{MR}).  The objects of the syntactic category are given by geometric formulae in context, which we write as $\phi(\vec{x})$, while an arrow $[\theta] \colon \phi(\vec{x}) \to \psi(\vec{y})$ is labelled by a \emph{provably functional formula} (up to provable equivalence modulo $\theory$; see \cite[\S D1.4]{elephant}, \cite[Definition 1.4.1]{TST}), i.e.\ a formula $\theta(\vec{x},\vec{y})$ such that $\theory$ proves the sequents 
	\begin{align*}
		\theta(\vec{x},\vec{y}) & \vdash \phi(\vec{x}) \land \psi(\vec{y}) & \text{(well-definedness)}, \\
		\phi(\vec{x})& \vdash \exists \vec{y} . \ \theta(\vec{x},\vec{y}) & \text{(totality)}, \\
		\theta(\vec{x},\vec{y}) \land \theta(\vec{x},\vec{y}') & \vdash \vec{y} = \vec{y}' & \text{(functionality)}.
	\end{align*}
	The composite of two such arrows $[\chi] \circ [\theta]$ is given by the relational composite, i.e.\ $\exists \vec{y}. \ \theta(\vec{x},\vec{y}) \land \chi(\vec{y},\vec{z})$.  
	
	Considered as an object of $\topos_\theory$, the formula $(y=y)$, i.e.\ the true formula with one free variable, is the \emph{pre-bound} of the topos $\topos_\theory$ (in the terminology of \cite[Definition D3.2.3]{elephant}) given by the choice of theory (see \cite[Theorem D3.2.5]{elephant}, and recall that we are assuming that $\theory$ is single-sorted).  Next, given a set $\Psi$ of geometric formulae, each $\psi(\vec{x}) \in \Psi$ yields an object of the syntactic category and hence of the classifying topos $\topos_\theory$.  We will therefore also use $\Psi$ to denote their coproduct $\coprod_{\psi \in \Psi} \psi(\vec{x})$ in $\topos_\theory$ (note that this object may land outside of the syntactic category).
	\begin{prop}\label{prop:etendue-implies-coord}
		Let $\theory$ be a geometric theory and let $\Psi$ be a set of formulae.  The datum of a uniform co-ordinatisation of $\theory$ over $\Psi$ is equivalent to a covering by subterminals of the product projection $\Psi \times (y=y) \to \Psi$, considered as an object of the slice topos $\topos_\theory / \Psi$.
	\end{prop}
	\begin{proof}
		Suppose that the product projection $\Psi \times (y=y) \to \Psi$ admits a covering by subterminals, i.e.\ there is a jointly epimorphic family of maps to $\Psi \times (y=y) \to \Psi$ whose domains are subterminals.  Note that the subterminals of $\topos_\theory / \Psi$ are the subobjects of $\Psi$ in $\topos_\theory$.  Without loss of generality, we can assume that the covering of $\Psi \times (y=y) \to \Psi$ by subterminals factors through the covering by subobjects
		\[\textstyle \lrset{\psi(\vec{x}) \land (y=y) \subseteq \coprod_{\psi \in \Psi} \psi(\vec{x}) \land (y=y) \cong \Psi \times (y=y)}{\psi \in \Psi}.\]  
		That is, for each $\psi \in \Psi$, we ask for a jointly epimorphic family of morphisms $\sett{[\theta_j]}{j \in J_\psi}$ making the diagram
		\[
		\begin{tikzcd}
			& \psi(\vec{x}) \land ( y = y) \ar{d} \ar[hook]{r} & \coprod_{\psi \in \Psi} \psi(\vec{x}) \land (y=y) \ar{d} \\
			\psi_j(\vec{x}) \ar{ru}{[\theta_j]} \ar[hook]{r} & \psi(\vec{x}) \ar[hook]{r} & \coprod_{\psi \in \Psi} \psi(\vec{x})
		\end{tikzcd}
		\]
		of morphisms in $\topos_\theory$ commute, where $\psi_j(\vec{x})$ describes a subobject of $\psi(\vec{x})$, i.e.\ a formula such that $\theory$ proves $\psi_j(\vec{x}) \vdash \psi(\vec{x})$ (see \cite[Lemma D1.4.4]{elephant}).  To be completely precise, as the syntactic category is a full subcategory of $\topos_\theory$, we are asking for a family of provably functional formulae $\sett{\theta_j(\vec{x},\vec{x}',y)}{j \in J_\psi}$ for which the diagram
		\begin{equation}\label{eq:subobj-diag}
			\begin{tikzcd}
				& \psi(\vec{x}') \land (y=y) \ar{d}{[\psi(\vec{x}') \land \psi(\vec{x}'') \land (y=y) \land \vec{x}' = \vec{x}'']} \\
				\psi_j(\vec{x}) \ar{ru}{[\theta_j(\vec{x},\vec{x}',y)]} \ar{r}[']{[\psi_j(\vec{x}) \land \vec{x} = \vec{x}'']} & \psi(\vec{x}'')
			\end{tikzcd}
		\end{equation}
		commutes (see \cite[Lemma D1.4.2]{elephant} for the syntactic description of the product projection) and which is moreover jointly epimorphic.  The commutativity of \cref{eq:subobj-diag} ensures that $\theory$ proves the sequent $\theta_j(\vec{x},\vec{x}',y) \vdash \vec{x} = \vec{x}'$, and so we can safely replace $\theta_j$ with the formula $\theta'_j(\vec{x},y) \equiv \theta_j(\vec{x},\vec{x},y)$.  In summary, for each $\psi \in \Psi$, there is a set of formulae $\sett{[\theta'_j]}{j \in J_\psi}$ for which $\theory$ proves the sequents
		\begin{align*}
			\theta'_j(\vec{x},y) & \vdash \psi(\vec{x}) & \text{(as $\theta_j$ is well-defined)}, \\
			\theta'_j(\vec{x},y) \land \theta'_j(\vec{x},y') & \vdash y = y' & (\text{as $\theta_j$ is functional}), \\
			\psi(\vec{x}) \land (y=y) & \vdash \textstyle \bigvee_{j \in J_\psi}  \theta'_j(\vec{x},y) & \text{(as $\sett{[\theta_j]}{j \in J_\psi}$ is jointly epimorphic)},
		\end{align*}
		and so we conclude that $\theory$ is uniformly co-ordinatised over $\Psi$.  For the converse, we simply reverse the above construction to observe that if $\theory$ is uniformly co-ordinatised over $\Psi$, then the product projection $\Psi \times (y =y ) \to \Psi \in \topos_\theory / \Psi$ is covered by subterminals.
	\end{proof}
	Recall that a topos is localic if and only if every object admits a covering by subterminals (\cite[Theorem 5.37]{topos}).  Thus, if $\topos_\theory/\Psi$ is localic, then in particular the product projection $\Psi \times (y=y) \to \Psi$ is covered by subterminals and so, by \cref{prop:etendue-implies-coord}, $\theory$ is uniformly co-ordinatised over $\Psi$.  This is a substantial component of the proof that the classifying topos $\topos_\theory$ is an \'etendue if and only if $\theory$ is uniformly co-ordinatisable.  The remaining steps involve combining facts about an arbitrary topos $\ftopos$.
	
	We first observe that if $\ftopos$ is an \'etendue over an object $A$, then $\ftopos$ is an \'etendue over any cover of $A$.  This is a consequence of the following fact concerning slice topoi, which effectively asserts that a slice topos $\ftopos/A$ is the topos of sheaves on the \emph{internal discrete locale} given by the object $A$.  Recall from \cite[\S 1]{fact1} that a geometric morphism $f \colon \mathcal{G} \to \ftopos$ is \emph{localic} if every object of $\mathcal{G}$ is a subquotient of an object in the image of the inverse image part $f^\ast$. 
	\begin{lem}\label{lem:etale_is_localic}
		An \'etale geometric morphism $\pi_A \colon \ftopos/A \to \ftopos$ is localic.
	\end{lem}
	\begin{proof}
		Given an object in $\ftopos/A$, i.e.\ an arrow $f \colon B \to A$ in $\ftopos$, then taking the \emph{graph} $G_f$ of $f$ yields the diagram
		\begin{equation}\label{eq:graph}
		\begin{tikzcd}
			B \ar{rd}[']{f} \ar[hook]{r}{G_f} & B \times A \ar{d} \\
			& A,
		\end{tikzcd}
		\end{equation}
		and so $f \in \ftopos/A$ is a subobject of an object in the inverse image $\pi_A^\ast \colon \ftopos \to \ftopos/A$.  Hence, $\pi_A$ is localic.
	\end{proof}
	\begin{coro}\label{coro:etendue_over_A}
		If $\ftopos$ is an \'etendue over an object $A \in \ftopos$, then for any epimorphism $B \twoheadrightarrow A$, the topos $\ftopos$ is also an \'etendue over $B$.
	\end{coro}
	\begin{proof}
		As $g \colon B \twoheadrightarrow A$ is an epimorphism, so is the composite $B \twoheadrightarrow A \twoheadrightarrow 1$.  Next, the \'etale geometric morphism $\ftopos/B \simeq (\ftopos/A)/g \to \ftopos/A$ is localic by \cref{lem:etale_is_localic}, and so $\ftopos/B$ admits a localic morphism to a localic topos and is therefore localic  (\cite[Theorem C1.4.7]{elephant}).
	\end{proof}

	Next, we note the following consequences of the fact that monomorphisms and jointly epimorphic families are stable under pullback in a topos (cf.\ \cite[\S A1.3]{elephant}).
	\begin{lem}\label{lem:pbs_of_cov}
		Let $\ftopos$ be any topos with objects $A, B \in \ftopos$.
		\begin{enumerate}
			\item\label{enum:subobjs_still_cov} If $A$ is covered by subterminals, so is any subobject $A' \subseteq A$;
			\item\label{enum:prod_still_cov} If $A, B$ are covered by subterminals, so is the product.
		\end{enumerate}
	\end{lem}
	\begin{proof}
		For \cref{enum:subobjs_still_cov}, if $\sett{f_i \colon U_i \to A}{i \in I}$ is a cover of $A$, i.e.\ a jointly epimorphic family of morphisms, whose domains are all subterminals $U_i \subseteq 1$, we obtain a cover of a subobject $A' \subseteq A$ by taking the pullbacks
		\[\begin{tikzcd}
			{U_i'} & {U_i} & 1 \\
			{A'} & A,
			\arrow[hook, from=1-1, to=1-2]
			\arrow["{f'_i}"', from=1-1, to=2-1]
			\arrow["\lrcorner"{anchor=center, pos=0.125}, draw=none, from=1-1, to=2-2]
			\arrow[hook, from=1-2, to=1-3]
			\arrow["{f_i}", from=1-2, to=2-2]
			\arrow[hook, from=2-1, to=2-2]
		\end{tikzcd}\]
		which remains a cover because jointly epimorphic families are stable under pullback, and the domains remain subterminals because monomorphisms are stable under pullback.
		
		For \cref{enum:prod_still_cov}, if $\sett{f_i \colon U_i \to A}{i \in I}$ and $\sett{g_j \colon V_j \to B}{j \in J}$ are covers with $U_i, V_j \hookrightarrow 1$ subterminals, then the family $\sett{f_i \times g_j \colon U_i \times V_j \to A \times B}{i \in I, j \in J}$, consisting of the universally induced maps
		\[\begin{tikzcd}
			{U_i \times V_j} & {U_i \times B} & {U_i} \\
			{A \times V_j} & {A \times B} & A \\
			{V_j} & B & 1,
			\arrow[from=1-1, to=1-2]
			\arrow[from=1-1, to=2-1]
			\arrow["\lrcorner"{anchor=center, pos=0.125}, draw=none, from=1-1, to=2-2]
			\arrow["{f_i \times g_j}", between={0.35}{1}, dashed, from=1-1, to=2-2]
			\arrow[from=1-2, to=1-3]
			\arrow[from=1-2, to=2-2]
			\arrow["\lrcorner"{anchor=center, pos=0.125}, draw=none, from=1-2, to=2-3]
			\arrow["{f_i}", from=1-3, to=2-3]
			\arrow[from=2-1, to=2-2]
			\arrow[from=2-1, to=3-1]
			\arrow["\lrcorner"{anchor=center, pos=0.125}, draw=none, from=2-1, to=3-2]
			\arrow[from=2-2, to=2-3]
			\arrow[from=2-2, to=3-2]
			\arrow["\lrcorner"{anchor=center, pos=0.125}, draw=none, from=2-2, to=3-3]
			\arrow[from=2-3, to=3-3]
			\arrow["{g_j}"', from=3-1, to=3-2]
			\arrow[from=3-2, to=3-3]
		\end{tikzcd}\]
		is still a cover by subterminals.  The families $\sett{f_i \times \id_B \colon U_i \times B \to A  \times B}{i \in I}$ and, for each $i \in I$, $\sett{\id_{U_i}\times g_j \colon U_i \times V_j \to A \times B}{j \in J}$ are jointly epimorphic by pullback stability, and hence so is the family of composites $\sett{f_i \times g_j}{i \in I, j \in J}$.   The domain $U_i \times V_j$ is a subterminal, indeed it is the meet of $U_i, V_j$ as subterminals since the outer square is a pullback.
	\end{proof}	
	\begin{prop}\label{prop:prebound_cov_implies_localic}
		Let $\ftopos$ be a topos with a pre-bound $P \in \ftopos$.  For an object $A \in \ftopos$, the slice $\ftopos / A$ is localic if and only if the product projection $P \times A \to A$, viewed as an object of $\ftopos / A$, is covered by subterminals.
	\end{prop}
	\begin{proof}
		One direction is immediate: if $\ftopos / A$ is localic, then in particular the product projection $P \times A \to A$ must be covered by subterminals.  For the converse, we first show that $P \times A \to A \in \ftopos / A$ being covered by subterminals implies that any product projection $B \times A \to A$ is covered by subterminals.  As $P$ is a pre-bound for $\ftopos$, every object $B$ admits a cover by subobjects of finite powers of $P$, i.e.\ there is a jointly epimorphic family of arrows $\sett{f_i \colon U_i \to B}{i \in I}$ with $U_i \subseteq P^{n_i}$.  Hence there is a jointly epimorphic family of arrows
		\[\begin{tikzcd}
			{U_i \times A} & {P^{n_i} \times A} \\
			{B \times A} & A.
			\arrow[hook, from=1-1, to=1-2]
			\arrow["{f_i \times \id_A}"', from=1-1, to=2-1]
			\arrow[from=1-2, to=2-2]
			\arrow[from=2-1, to=2-2]
		\end{tikzcd}\]
		The product projection $U_i \times A \to A$ is a subobject of $P^{n_i} \times A \to A$, which in turn is the $n_i$-fold product of $P \times A \to A \in \ftopos / A$.  Thus, by \cref{lem:pbs_of_cov} the projection $U_i \times A \to A$ is covered by subterminals, and hence so is $B \times A \to A$.
		
		Now note that for an arbitrary object $f \colon B \to A$ of $\ftopos/A$, by taking the graph of $f$ as in \cref{eq:graph}, $f \colon B \to A$ is a subobject of the product projection $B \times A \to A$, and so by a further application of \cref{lem:pbs_of_cov} we have that $f \colon B \to A$ is covered by subterminals.  Thus, $\ftopos / A$ is localic as desired.
	\end{proof}
	Finally, we can combine the above results to deduce the following, which completes the equivalence of conditions \cref{enum:coord} and \cref{enum:etendue} in our main theorem.
	\begin{coro}\label{coro:etendue_implies_coord_over_set}
		Let $\theory$ be a geometric theory.  The classifying topos $\topos_\theory$ is an \'etendue if and only if $\theory$ is uniformly co-ordinatised over some inhabited set of formulae $\Psi$.
	\end{coro}
	\begin{proof}
		Suppose that $\topos_\theory$ is an \'etendue over an inhabited object $A \in \topos_\theory$.  Since the objects of the form $\phi(\vec{x})$ generate the topos $\topos_\theory$, the object $A$ admits an epimorphism $\coprod_{\psi \in \Psi} \psi(\vec{x}) \twoheadrightarrow A$ for some family of geometric formulae $\Psi$.  Thus, the composite $\coprod_{\psi \in \Psi} \psi(\vec{x}) \twoheadrightarrow A \twoheadrightarrow 1$ is an epimorphism, from which we deduce that $\theory$ proves the sequent $\top \vdash \bigvee_{\psi \in \Psi} \exists \vec{x}. \ \psi(\vec{x})$.  Additionally, by \cref{coro:etendue_over_A}, $\topos_\theory$ is an \'etendue over the coproduct $\Psi \cong \coprod_{\psi \in \Psi} \psi(\vec{x})$, in particular $\topos_\theory{/}\Psi$ is localic, and so $\theory$ is uniformly co-ordinatised over $\Psi$ by \cref{prop:etendue-implies-coord}.
		
		Conversely, if $\theory$ is uniformly co-ordinatised over an inhabited set of geometric formulae $\Psi$, then $\Psi$ considered as an object of $\topos_\theory$ is inhabited, and moreover the slice $\topos_\theory/\Psi$ is a localic topos by \cref{prop:etendue-implies-coord} and \cref{prop:prebound_cov_implies_localic}.  Hence, $\topos_\theory$ is an \'etendue as desired.
	\end{proof}
	\begin{ex}[Continuous flat functors on a monic site]
		Let $(\cat,J)$ be a \emph{monic site}, i.e.\ a category $\cat$ in which every arrow is a monomorphism equipped with a Grothendieck topology $J$.  In \cite[\S 1]{rosenthal_monic}, it is shown that the topos of sheaves $\Sh(\cat,J)$ on a monic site is an étendue, while the converse is shown in \cite{mono_site}.  The topos $\Sh(\cat,J)$ classifies the geometric theory of \emph{continuous flat functors on} $(\cat,J)$ (see \cite[Theorem 2.1.11]{TST}).  As a single-sorted theory, this has the following description:
		\begin{enumerate}
			\item\label{enum:flat:inhab} For each object $c \in \cat$, we add a unary relation symbol $R_c(x)$ and the axioms
			\[\textstyle
			(x=x) \vdash \bigvee_{c \in \cat} R_c(x), \quad \top \vdash \bigvee_{c \in \cat} \exists x. \ R_c(x).
			\]
			\item For each arrow $f \colon c \to d$ in $\cat$, we add a binary relation and axioms expressing the graph of a function from $R_c(x)$ to $R_d(y)$.  We will abuse notation and denote this binary relation by $\overline{f}(x) = y$, as though it derived from a function symbol, and similarly write $\overline{f}(x) = \overline{g}(z)$ for $\exists y. \ ( \overline{f}(x) = y \land y = \overline{g}(z))$, etc.  We also add the axioms
			\[
			R_c(x) \vdash \overline{\id_c}(x) = x, \quad R_c(x) \vdash \overline{h} \ \overline{f} (x) = \overline{h \circ f}(x),
			\]
			for every identity arrow and every pair of composable arrows $f \colon c \to d, h \colon d \to e$ in $\cat$.
			\item\label{enum:flat:flat} For every pair of objects $a, b \in \cat$ and every pair of parallel arrows $f, g \colon a \to b$ in $\cat$, we add the axioms
			\begin{align*}
				R_a(x) \land R_b(y) &  \vdash \bigvee_{\substack{f : c \to a , \\ g : c \to b \, \in \, \cat}} \exists z. \ ( \overline{f}(z) = x \land \overline{g}(z) = y), \\
				\overline{f}(x) = \overline{g}(x) & \vdash \bigvee_{\substack{h : c \to a \, \in \, \cat \\ \text{s.t.} \, f \circ h = g \circ h}} \exists z. \ (  \overline{h}(z) = x).
			\end{align*}
			\item For every $J$-covering sieve $S$ on $c \in \cat$, we add the geometric axiom
			\[\textstyle 
			R_c(x) \vdash \bigvee_{f \in S} \exists y. \ \overline{f}(y) = x.
			\]
		\end{enumerate}
		Note that, since every map in $\cat$ is a monomorphism, and those finite limits that exist in $\cat$ are preserved by flat functors, the image of every arrow in $\cat$ under a flat functor is still a monomorphism, and thus the theory of continuous flat functors proves the sequent $\overline{f}(x) = \overline{f}(y) \vdash x = y$ for every $f \in \cat$.
		
		Equipped with this axiomatisation, we can explicitly demonstrate that the theory of continuous flat functors on a monic site $(\cat,J)$ is uniformly co-ordinatised.  We take the set of geometric formulae $\Psi $ over which the theory is uniformly co-ordinatised to be $ \sett{R_a(x)}{a \in \cat}$, which is inhabited by the second axiom in \cref{enum:flat:inhab}.  For each $R_a(x) \in \Psi$, the family of geometric formulae $\Theta_{R_a(x)}$ witnessing the uniform co-ordinatisation is given by 
		\[\sett{R_a(x) \land R_b(y) \land \exists z. \ ( \overline{f}(z) = x \land \overline{g} (z) = y)}{b \in \cat, f \colon c \to a \in \cat, g \colon c \to b \in \cat}.\]
		Indeed, we have the necessary derivations required by \cref{df:coord}, e.g.\ the sequent
		\[
		R_a(x) \land (y=y) \vdash \bigvee_{b \, \in \, \cat} \bigvee_{\substack{f : c \to a , \\ g : c \to b \, \in \, \cat}} R_a(x) \land R_b(y) \land \exists z. \ ( \overline{f}(z) = x \land \overline{g} (z) = y)
		\]
		follows from the first axioms in \cref{enum:flat:inhab} and \cref{enum:flat:flat}, while that the formula
		\begin{multline*}
			(R_a(x) \land R_b(y) \land \exists z. \ (  \overline{f}(z) = x \land \overline{g} (z) = y)) \\ \land (R_a(x) \land R_b(y') \land \exists z'. \ ( \overline{f}(z') = x \land \overline{g} (z') = y'))
		\end{multline*}
		entails $y =y'$ follows from the fact that the theory of continuous flat functors on a monic site proves that $\overline{f}(z) = x = \overline{f}(z') \vdash z = z' \vdash y = \overline{g}(z) = \overline{g}(z') = y'$.
	\end{ex}
	\section{\'Etendues and uniform rigidity}\label{sec:rigid}
	Having demonstrated the equivalence between our syntactic condition, uniform co-ordination, and the property of having an \'etendue as a classifying topos, we turn to describing our equivalent semantic condition.  Since a theory $\theory$ that is uniformly co-ordinatised over a set of formulae $\Psi$ is, in particular, co-ordinatised over $\Psi$, if follows by Gaifman's co-ordinatisation theorem that every (set-based) model $M \vDash \theory$ is rigid over $\Psi(M)$.  Given the special nature of a uniform co-ordinatisation, it is not too surprising that $M$ is rigid in a stronger sense.
	\begin{ex}\label{ex:rigid_vect}
		Consider again the theory $\Vectff$ of finite dimensional $F$-vector spaces. In addition to the discussion in \cref{ex:vectff}, there is another way that a model $V \vDash \Vectff$, satisfying $V \vDash \exists \vec{x}. \ B_n(\vec{x})$, is determined by any witness $(v_1,\dots,v_n) \in V^n$ of $B_n(\vec{x})$.   Given any other $W \vDash \Vectff$ with $W \vDash B_n(w_1, \dots , w_n)$, there is a unique isomorphism $V \cong W$ that sends $(v_1, \dots , v_n) \in V^n$ pointwise to $(w_1, \dots , w_n) \in W^n$.  Our notion of \emph{uniform rigidity} abstracts this property.
	\end{ex}
	\begin{df}\label{df:rigid} Let $\theory$ be a geometric theory, $\Psi$ be a set of geometric formulae over its language, and let $M \vDash \theory$ be a model.
		\begin{enumerate}
			\item We say that $M$ is \emph{uniformly rigid} over $\Psi$ if, for any element $\vec{m} \in \Psi(M)$, any other model $M' \vDash \theory$ and $\vec{m}' \in \Psi(M')$, there is at most one isomorphism of models $M \cong M'$ that maps $\vec{m}$ pointwise to $\vec{m}'$.  In other words, the only automorphism of the model $M$ that fixes $\vec{m}$ pointwise is the identity.
			\item We shall say that $M$ is \emph{strongly uniformly rigid} over $\Psi$ if, for any element $\vec{m} \in \Psi(M)$, any other model $M' \vDash \theory$ and element $\vec{m}' \in \Psi(M')$, there is at most one homomorphism of models $M \to M'$ that maps $\vec{m}$ pointwise to $\vec{m}'$.
		\end{enumerate}
	\end{df}
	We are required to consider \emph{strong} uniform rigidity because we are working in a positive setting, as the following example illustrates.  However, when restricting to the classical setting, the distinction between uniform rigidity and strong uniform rigidity is erased (\cref{prop:classical_strong_rigid_is_rigid}).
	\begin{ex}
		Clearly, strong uniform rigidity implies uniform rigidity.  The converse is not true, as demonstrated in the following example suggested by the anonymous reviewer.  
		Take the theory consisting of one unary function symbol $\sigma$ and the single axiom $(x=x) \land (y=y) \vdash x = y \lor x= \sigma(y) \lor y = \sigma(x)$.  The model $\{0,1\}$ with $\sigma(0) = \sigma(1) = 0$ is uniformly rigid but not strongly uniformly rigid.
	\end{ex}
	\begin{rem}\label{rem:model_in_any_topos}
		The notion of uniform rigidity given in \cref{def:relcat} makes sense in any topos $\ftopos$.  Given a model $M \vDash \theory$ internal to $\ftopos$ (see \cite[\S D1.2]{elephant} for how to interpret $\theory$ inside an arbitrary topos), we can still construct the coproduct $\Psi(M)$, and we take $\vec{m} \in \Psi(M)$ to mean a \emph{global element} $\vec{m} \colon 1 \to \Psi(M)$.  Note that, in an arbitrary topos, such a global element need not necessarily factor through any of the coproduct inclusions $\psi(M) \hookrightarrow \Psi(M)$.  As a simple example, for a set $I$, a model of $\Vectff$ in $\Set^I$ consists of an $I$-indexed family of finite dimensional $F$-vector spaces $\{V_i\}_{i \in I}$, but there is no need for $V_i$ and $V_{i'}$ to have the same dimension.
	\end{rem}
	\begin{df}\label{def:relcat}
		Thus, we will say that a geometric theory $\theory$ is \emph{uniformly rigid} over $\Psi$ in a topos $\ftopos$ (respectively, \emph{strongly uniformly rigid}) if every model internal to $\ftopos$ is (strongly) uniformly rigid over $\Psi$.  We say that $\theory$ is (\emph{strongly}) \emph{uniformly rigid} in $\ftopos$, without qualification, if there exists such a set of geometric formulae $\Psi$, which is inhabited, for which $\theory$ is (strongly) uniformly rigid over $\Psi$ in $\ftopos$.
	\end{df}
	We will see that, given a \emph{coherent} (respectively, \emph{classical}) theory $\theory$, its classifying topos $\topos_\theory$ is an \'etendue if and only if $\theory$ is strongly uniformly rigid (resp., uniformly rigid) in the topos of sets (\cref{thm:mainthm_coherent} and \cref{coro:mainthm_classical}).  For an analogous result concerning arbitrary geometric theories, we will require further conditions, as elaborated below.
	\subsection{Types in a topos}\label{sec:types}
	We abstract the notion of a model-theoretic type to the topos-theoretic setting.  Recall that, in the set-based setting, an $n$-type $\type$ of a geometric theory $\theory$ consists of a \emph{completely prime filter} of the frame of geometric formulae with $n$ free variables.  Given a formula in context $\psi(\vec{x})$ (with $n$ free variables say), we make the further distinction and define a $\psi(\vec{x})$-type to be an $n$-type which contains $\psi(\vec{x})$.  Note that these bijectively correspond to completely prime filters of the frame of geometric formulae in $n$ free variables that prove $\psi(\vec{x})$.	Recall also that such a $\psi(\vec{x})$-type is \emph{realised} in a set-based model if there is a model $M \vDash \theory$ and an $n$-tuple $\vec{m} \in \psi(M)$ such that $\type = \sett{\phi(\vec{x})}{M \vDash \phi(\vec{m})} $.  Similarly, given a set of geometric formulae $\Psi$, by a $\Psi$-type we mean a $\psi(\vec{x})$-type for some $\psi(\vec{x}) \in \Psi$, and say that a $\Psi$-type is realised if it is realised as a $\psi(\vec{x})$-type in the sense above. 
	
	Note that the frame of geometric formulae with $n$ free variables that prove $\psi(\vec{x})$ is just the subobject lattice $\Sub_{\topos_\theory}(\psi(\vec{x}))$, when $\psi(\vec{x})$ is viewed as an object of the classifying topos $\topos_\theory$ (\cite[Lemma D1.4.4]{elephant}).  Thus, we can equate $\psi(\vec{x})$-types with frame homomorphisms $ \Sub_{\topos_\theory}(\psi(\vec{x})) \to 2$ to the two element frame $2 = \{\bot,\top\}$, since these correspond to the completely prime filters on $\Sub_{\topos_\theory}(\psi(\vec{x}))$.  Given a model $M \vDash \theory$, by taking the interpretation of geometric formulae in $M$ we obtain a frame homomorphism $\Sub_{\topos_\theory}(\psi(\vec{x})) \to \power(\psi(M))$ to the powerset of $\psi(M)$.  Recall that the $n$-tuples $\vec{m} \in \psi(M)$ correspond to the frame homomorphisms $\power(\psi(M)) \to 2$, where a tuple $\vec{m} \in \psi(M)$ is identified with the map that sends $U \subseteq \psi(M)$ to $\top \in 2$ if $\vec{m} \in U$, and $\bot$ otherwise.  Therefore, a $\psi(\vec{x})$-type, written as a frame homomorphism $\type \colon \Sub_{\topos_\theory}(\psi(\vec{x})) \to 2$, is realised in a model $M \vDash \theory$ if and only if it factors as
	\[\begin{tikzcd}
		\Sub_{\topos_\theory}(\psi(\vec{x})) \ar{r} \ar[bend left]{rr}{\type} & \power(\psi(M)) \ar[dashed]{r} & 2.
	\end{tikzcd}\]
	Analogously, a $\Psi$-type corresponds to a frame homomorphism $\type \colon \Sub_{\topos_\theory}(\Psi) \to 2$, and $\type$ is realised in a model $M$ if and only if $\type$ factors through the interpretation morphism $\Sub_{\topos_\theory}(\Psi) \to \power(\Psi(M))$.  Thus, the following is a natural generalisation of the notion of model-theoretic type to an arbitrary topos.
	\begin{df}\label{df:realise-types}
		Let $\theory$ be a geometric theory with classifying topos $\topos_\theory$, and let $\Psi$ be a set of geometric formulae.  
		\begin{enumerate}
			\item Given a topos $\ftopos$, by a $\Psi$-\emph{type in} $\ftopos$, we mean a frame homomorphism 
			\[\type  \colon \Sub_{\topos_\theory}(\Psi) \to \Sub_\ftopos(1).\]
			\item We say that the $\Psi$-type $\type$ in $\ftopos$ is \emph{realised} if there exists a geometric morphism $M \colon \ftopos \to \topos_\theory$ (i.e.\ a model of $\theory$ internal to $\ftopos$) and a global element ${m} \colon 1 \to M^\ast \Psi$ such that $\type$ is the composite frame homomorphism
			\[
			\begin{tikzcd}
				\Sub_{\topos_\theory}(\Psi) \ar{r}{M^\ast} & \Sub_\ftopos(M^\ast \Psi) \ar{r}{{m}^\ast} & \Sub_\ftopos(1),
			\end{tikzcd}
			\]
			where $M^\ast \colon \Sub_{\topos_\theory}(\Psi) \to \Sub_\ftopos(M^\ast \Psi)$ is inverse image of $M$ acting on monomorphisms, and ${m}^\ast \colon \Sub_\ftopos(M^\ast \Psi) \to \Sub_\ftopos(1)$ denotes the homomorphism that sends a subobject $U \subseteq M^\ast \Psi$ to its pullback along ${m} \colon 1 \to M^\ast \Psi$.
		\end{enumerate}
	\end{df}
	Since we are working within a \emph{positive} context, the model-theoretic types are not discretely ordered.  That is, there could exist a pair $\type, \type'$ of distinct $\psi(\vec{x})$-types where $\type$ is properly contained in $\type'$.  Suppose that $\type$ and $\type'$ are $\psi(\vec{x})$-types that are realised by tuples $\vec{m} \in M \vDash \theory$ and $\vec{n} \in N \vDash \theory$.  If $f \colon M \to N$ is a homomorphism of models for which $f(\vec{m}) = \vec{n}$, then since $f$ preserves satisfaction of formulae, it follows that $\type \subseteq \type'$.  We therefore say that the inclusion $\type \subseteq \type'$ is \emph{realised} by the homomorphism $f$, or in the context of an arbitrary topos:
	\begin{df}
		Given a topos $\ftopos$, let $\type, \type' \colon \Sub_{\topos_\theory}(\Psi) \to \Sub_{\ftopos}(1)$ be a pair of $\Psi$-types in $\ftopos$. 
		\begin{enumerate}
			\item We say that $\type$ is \emph{included} in $\type'$ if $\type(U) \leqslant \type'(U)$ for all $U \in \Sub_{\topos_\theory}(\Psi)$ (i.e.\ $\type \leqslant \type'$ in the \emph{specialisation order} on the poset $\Frm(\Sub_{\topos_\theory}(\Psi) , \Sub_{\ftopos}(1)) $).
			\item Suppose that $\type, \type'$ are realised, in the sense of \cref{df:realise-types}, by respective geometric morphisms $M, N \colon \ftopos \to \topos_\theory$ and global elements $m \colon 1 \to M^\ast \Psi$ and $n \colon 1 \to N^\ast \Psi$.   We say that an inclusion $\type \subseteq \type'$ is \emph{realised} in $\ftopos$ if there is a natural transformation $f \colon M^\ast \Rightarrow N^\ast$ (i.e.\ a homomorphism $f \colon M \to N$ of the corresponding internal models) such that $f$ maps $m$ onto $n$, i.e.\ the diagram
			\[
			\begin{tikzcd}[row sep=tiny]
				& M^\ast \Psi \ar{dd}{f_{\Psi}} \\
				1 \ar[bend left = 20]{ru}{m} \ar[bend right = 20]{rd}[']{n} & \\
				& N^\ast \Psi
			\end{tikzcd}
			\]
			commutes.
		\end{enumerate}
	\end{df}
	\subsection{\'Etendues are uniformly rigid}
	We now prove that the classifying topos of a geometric theory is an \'etendue if and only if there is an inhabited set of geometric formulae over which the theory is both strongly uniformly rigid and realises all types and inclusions of types (in any topos).
	Recall that, for a topos $\ftopos$, the unique geometric morphism $\gamma \colon \ftopos \to \Set$ factorises (uniquely up to isomorphism) as a hyperconnected morphism followed by a localic morphism $\ftopos \to \Sh(\Sub_\ftopos(1)) \to \Set$, and moreover the topos $\ftopos$ is localic if and only if the hyperconnected part $\ftopos \to \Sh(\Sub_\ftopos(1))$ is an equivalence (see \cite[\S A4.6]{elephant}; the morphism $\ftopos \to \Sh(\Sub_\ftopos(1))$ is the unit of the \emph{localic reflection}).  
	
	In particular, if $\theory$ is a geometric theory and $\Psi$ is a set of geometric formulae, then we obtain a hyperconnected morphism $\topos_\theory / \Psi \to \Sh(\Sub_{\topos_\theory / \Psi}(1)) \simeq \Sh(\Sub_{\topos_\theory}(\Psi))$, which we denote by $h \colon \topos_\theory / \Psi \to \Sh(\Sub_{\topos_\theory}(\Psi))$.
	\begin{prop}\label{prop:psi_types}
		Let $\theory$ be a geometric theory and $\Psi$ a collection of geometric formulae.  For a given topos $\ftopos$,
		\begin{enumerate}
			\item\label{enum:psi_types_are_geom} the poset of $\Psi$-types in $\ftopos$ is given by the category of geometric morphisms $\Topos(\ftopos,\Sh(\Sub_{\topos_\theory}(\Psi)))$,
			\item\label{enum:psi_types_realised} every $\Psi$-type in $\ftopos$ is realised if and only if the functor
			\[
			h \circ - \colon \Topos(\ftopos,\topos_\theory/\Psi) \to \Topos(\ftopos,\Sh(\Sub_{\topos_\theory}(\Psi))),
			\]
			induced by post-composition with the morphism $h \colon \topos_\theory/\Psi \to \Sh(\Sub_{\topos_\theory}(\Psi))$, is surjective on objects,
			\item\label{enum:inclsuibs_realised} every inclusion of $\Psi$-types in $\ftopos$ is realised if and only if the functor $h \circ -$ is full,
			\item\label{enum:h_faithful_iff_rigid} and $\theory$ is strongly uniformly rigid over $\Psi$ in the topos $\ftopos$ if and only if $h \circ -$ is faithful.
		\end{enumerate}
	\end{prop}
	\begin{proof}
		For \cref{enum:psi_types_are_geom}, recall that $ \Topos(\ftopos,\Sh(\Sub_{\topos_\theory}(\Psi))) \simeq \Frm(\Sub_{\topos_\theory}(\Psi),\Sub_\ftopos(1))$ (see, for instance, \cite[Proposition IX.5.3]{SGL}), where the latter is precisely the poset of $\Psi$-types in $\ftopos$ and their inclusions.  
		
		For \cref{enum:psi_types_realised}, recall from \cite[Proposition 5.12]{SGA4-3} (see \cite{jens} for an exposition in English) that a point $\ftopos \to \topos_\theory/\Psi$ corresponds to a pair $(M \colon \ftopos \to \topos_\theory, m \colon 1 \to M^\ast \Psi)$, i.e.\ a model $M \vDash \theory$ in $\ftopos$ and a global element of $\Psi(M)$.  The functor $h \circ - $ sends such a point to the frame homomorphism given by the composite
		\[
		\Sub_{\topos_\theory}(\Psi) \xrightarrow{M^\ast} \Sub_\ftopos(M^\ast \Psi) \xrightarrow{m^\ast} \Sub_\ftopos(1).
		\]
		Thus, the functor $h \circ -$ is surjective on objects if and only if every frame homomorphism $\Sub_{\topos_\theory}(\Psi) \to \Sub_\ftopos(1)$ is obtained from some pair $(M \colon \ftopos \to \topos_\theory, m \colon 1 \to M^\ast \Psi)$ as above, i.e.\ if and only if every $\Psi$-type in $\ftopos$ is realised.  
		
		Next, \cref{enum:inclsuibs_realised} follows by a similar argument.  Given a pair of points 
		\begin{equation}\label{eqn:pair_of_points}
		(M \colon \ftopos \to \topos_\theory, m \colon 1 \to M^\ast \Psi), (N \colon \ftopos \to \topos_\theory, n \colon 1 \to N^\ast \Psi) \colon \ftopos \rightrightarrows \topos_\theory / \Psi,
	\end{equation}
		we recall that a transformation between these points consists of a transformation $f \colon M^\ast \Rightarrow N^\ast$ such that the diagram
		\begin{equation}\label{eqn:homo_of_points}
		\begin{tikzcd}[row sep=tiny]
			& M^\ast \Psi \ar{dd}{f_{\Psi}} \\
			1 \ar[bend left = 20]{ru}{m} \ar[bend right = 20]{rd}[']{n} & \\
			& N^\ast \Psi
		\end{tikzcd}
		\end{equation}
		commutes.  By the naturality of $f$, or equivalently expressed the fact that the homomorphism $f$ preserves satisfaction of formulae, we have an inequality in the specialisation order 
		\[\begin{tikzcd}[row sep=tiny]
			& {\Sub_\ftopos(M^\ast \Psi)} \\
			{\Sub_{\topos_\theory}(\Psi)} \\
			& {\Sub_\ftopos(N^\ast \Psi)},
			\arrow["{f_\Psi^\ast}"', to=1-2, from=3-2]
			\arrow[""{name=0, anchor=center, inner sep=0}, "{M^\ast}", bend left = 20, from=2-1, to=1-2]
			\arrow["{N^\ast}"', bend right = 20, from=2-1, to=3-2]
			\arrow["{\large\leqslant}"{marking, allow upside down}, draw=none, shift right=2, to=3-2, from=0]
		\end{tikzcd}\]
		where $f_\Psi^\ast$ is the map that takes the pullback of a subobject along $f_\Psi \colon M^\ast \Psi \to N^\ast \Psi$.
		The functor $h \circ -$ sends such the transformation $f \colon (M \colon \ftopos \to \topos_\theory, m \colon 1 \to M^\ast \Psi) \Rightarrow (N \colon \ftopos \to \topos_\theory, n \colon 1 \to N^\ast \Psi)$ to the inequality $m^\ast M^\ast \leqslant n^\ast N^\ast$ in the pasting
		\[\begin{tikzcd}[row sep=tiny]
			& {\Sub_\ftopos(M^\ast \Psi)} \\
			{\Sub_{\topos_\theory}(\Psi)} && {\Sub_\ftopos(1)}.\\
			& {\Sub_\ftopos(N^\ast \Psi)}
			\arrow["{f_\Psi^\ast}"', to=1-2, from=3-2]
			\arrow[""{name=0, anchor=center, inner sep=0}, "{M^\ast}", bend left = 20, from=2-1, to=1-2]
			\arrow["{N^\ast}"', bend right = 20, from=2-1, to=3-2]
			\arrow["{\large\leqslant}"{marking, allow upside down}, draw=none, shift right=2, to=3-2, from=0]
			\arrow["{m^\ast}", bend left = 20, from=1-2, to=2-3]
			\arrow["{n^\ast}"', bend right = 20, from=3-2, to=2-3]
		\end{tikzcd}\]
		Thus, the functor $h \circ -$ is full if and only if every inequality between $m^\ast M^\ast \leqslant n^\ast N^\ast$ is obtained from a transformation $f \colon M^\ast \Rightarrow N^\ast$ such that $f_\Psi \circ m = n$, i.e.\ if and only if the inclusions of realised $\Psi$-types are also realised.
		
		Finally, for \cref{enum:h_faithful_iff_rigid}, note that since $\Topos(\ftopos,\Sh(\Sub_{\topos_\theory}(\Psi)))$ is a poset, the functor $h \circ -$ is faithful if and only if, for any pair of points as in \cref{eqn:pair_of_points}, there is at most one transformation of these points as in \cref{eqn:homo_of_points}.  But this precisely expresses the assertion that, for any pair of models $M, N \vDash \theory$ with global elements $m \colon 1 \to \Psi(M)$ and $n \colon 1 \to \Psi(N)$, there is at most one homomorphism $f \colon M \to N$ that sends $m$ to $n$, i.e.\ $\theory$ is strongly uniformly rigid over $\Psi$ in the topos $\ftopos$.
	\end{proof}
	\begin{thm}[Characterisation of \'etendues]\label{thm:mainthm_geometric}
		Let $\topos$ be a topos and let $\theory$ be a geometric theory classified by $\topos$.  The following are equivalent:
		\begin{enumerate}[label = (\Alph*)]
			\item\label{enum:mainthm_geom:coord} there exists an inhabited set of geometric formulae $\Psi$ such that $\theory$ is uniformly co-ordinatised over $\Psi$;
			\item\label{enum:rigid_and_types} there exists an inhabited set of geometric formulae $\Psi$ such that, for every topos $\ftopos$, the theory $\theory$ is strongly uniformly rigid over $\Psi$ in $\ftopos$, and all $\Psi$-types and their inclusions in $\ftopos$ are realised;
			\item\label{enum:mainthm_geom:etendue} $\topos$ is an \'etendue.
		\end{enumerate}
	\end{thm}
	\begin{proof}
		The equivalence \cref{enum:mainthm_geom:coord} $\iff$ \cref{enum:mainthm_geom:etendue} was proved in \cref{coro:etendue_implies_coord_over_set}.  For the equivalence \cref{enum:rigid_and_types} $\iff$ \cref{enum:mainthm_geom:etendue}, recall from \cref{coro:etendue_over_A} (cf.\ also the proof of \cref{coro:etendue_implies_coord_over_set}) that $\topos$ is an \'etendue if and only if there is some inhabited set of geometric formulae $\Psi$ for which $\topos/\Psi$ is localic.  By the above discussion, the topos $\topos/\Psi$ is localic if and only if the canonical morphism $h \colon \topos/\Psi \to \Sh(\Sub_\topos(\Psi))$ to its localic reflection is an equivalence.  The geometric morphism $h$ is an equivalence if and only if, for any topos $\ftopos$, post-composing with $h$, i.e.\ the functor
		\[
		h \circ - \colon \Topos(\ftopos,\topos/\Psi) \to \Topos(\ftopos,\Sh(\Sub_{\topos}(\Psi))),
		\]
		is an equivalence of categories.  By \cref{prop:psi_types}, the functor $h \circ -$ is essentially surjective if and only if every $\Psi$-type in $\ftopos$ is realised, full if and only if every inclusion of $\Psi$-types in $\ftopos$ is realised, and faithful if and only if $\theory$ is strongly uniformly rigid over $\Psi$  in $\ftopos$, thus completing the proof.	
	\end{proof}
	\begin{rem}
		The condition regarding realisation of types and their inclusions in \cref{enum:rigid_and_types} in \cref{thm:mainthm_geometric} is necessary: Example D3.1.15 in \cite{elephant} gives an example of a topos $\topos$ where, for any other topos $\ftopos$, the category $\Topos(\ftopos,\topos)$ is a preorder, but nonetheless $\topos$ is not localic.  Thus, this topos $\topos$ is strongly uniformly rigid over the terminal object $1 \in \topos$, but it is not an \'etendue (over $1$).  We can, however, drop the condition regarding types, and more, if our topos is known to be coherent (see \cref{thm:mainthm_coherent}).
	\end{rem}
	\begin{ex}[G-torsors]\label{ex:torsors}
		Recall from \cite[\S VIII]{SGL} that the \'etendue $\Set^G$ classifies $G$-\emph{torsors}, i.e.\ the free and transitive $G$-actions on inhabited objects (in the topos $\Sh(X)$ on a space, these are precisely the \emph{principal $G$-bundles}).  Explicitly, this is the single-sorted theory with a unary function symbol $\overline{g}$ for each $g \in G$ and the axioms:
		\[
		\top \vdash\overline{e} x = x, \quad \top  \vdash \overline{g} \, \overline{g}' x = \overline{(gg')}x, \quad \top  \vdash \exists x. \ (x=x), \] \[
		(y=y) \land (x=x)  \vdash \bigvee_{g \in G} y = \overline{g}x, \ \text{ and } \ 
		\overline{g} x = x  \vdash \bot \ \text{for all $g \in G\setminus e$.}
		\]
		The $G$-torsors in $\Set$ are all essentially the set $G$ with its canonical $G$-action, but where we have `forgotten' which element was the identity element.  As soon as one element of the set $G$ is fixed, the rest of the structure is fixed too, i.e.\ the set-based $G$-torsors are (strongly) uniformly rigid over the formula $(x=x)$.  The same remains true for $G$-torsors in an arbitrary topos.  Indeed, the formulae $\sett{(y = \overline{g}x)}{g \in G}$ uniformly co-ordinate the theory of $G$-torsors over $(x=x)$.
	\end{ex}
	\begin{ex}
		We have seen already that the theory $\Vectff$ of finite dimensional vector spaces over the finite field $F$ is uniformly co-ordinatised.  Thus, the classifying topos of $\Vectff$ is an \'etendue.  Indeed, we can recognise that the classifying topos of ${\Vectff}$ is given by the presheaves on the groupoid $\coprod_{n \in \N} \text{GL}_n(F)$, i.e.\ the disjoint union of the general linear groups over $F$ (for instance, using the techniques established in \cite{myselfgrpds}), and presheaves on groupoids are étendues.  Therefore, $\Vectff$ is strongly uniformly rigid over the set of formulae $\sett{B_n(\vec{x})}{n \in \N}$, as we would expect (cf.\ \cref{ex:rigid_vect}).
	\end{ex}
	\begin{rem}[Relative categoricity]
		A geometric theory $\theory$ satisfying the conditions of \cref{thm:mainthm_geometric} evidently also satisfies a uniform variant of \emph{rigid, relative categoricity}, strongly resembling \cite[Proposition 7.3]{relcat}.  Namely, let $M \vDash \theory$ and $N \vDash \theory$ be a pair of models in a topos $\ftopos$, and suppose we are given (global) elements $\vec{m} \in \Psi(M)$ and $\vec{n} \in \Psi(N)$ such that the $\Psi$-type of $\vec{m}$ is included in the $\Psi$-type of $\vec{n}$, then there is a unique homomorphism of models $f \colon M \to N$ that sends $\vec{m}$ pointwise to $\vec{n}$.
		
		If $\vec{m}$ and $\vec{n}$ have the same $\Psi$-type, then there is also a homomorphism $g \colon N \to M$ sending $\vec{m}$ pointwise to $\vec{n}$.  Thus, by uniform rigidity over $\Psi$, the homomorphism $g$ is an inverse to $f$, and so the two models $M, N$ are isomorphic.  Thus, any model of $\theory$ is determined, up to isomorphism, by the type of any witness of $\Psi$.
	\end{rem}
	\subsection{Coherent and classical theories}\label{sec:coherent}
	If we restrict to coherent theories, or classical theories, the conditions of \cref{thm:mainthm_geometric} simplify, giving a closer parallel to Gaifman's original co-ordinatisation result.
	\begin{thm}[Characterisation of coherent \'etendues]\label{thm:mainthm_coherent}
		Let $\topos$ be a coherent topos, and let $\theory$ be a coherent theory classified by $\topos$.  The following are equivalent:
		\begin{enumerate}[label = (\Alph*)]
			\item $\theory$ is uniformly co-ordinatised,
			\item $\theory$ is strongly uniformly rigid in sets,
			\item $\topos$ is an \'etendue.
		\end{enumerate}
	\end{thm}
	\begin{proof}
		By \cref{thm:mainthm_geometric}, we only need to show that if $\theory$ is strongly uniformly rigid in sets, then $\topos$ is an \'etendue.  For this implication, we apply \cite[Corollary D3.5.11]{elephant}, which asserts that a coherent topos $\ftopos$ is localic if and only if its category of set-points $\Topos(\Set,\ftopos)$ is a preorder.  Let $\Psi$ be an inhabited set of (geometric) formulae over which $\theory$ is strongly uniformly rigid in sets.  Initially, we show that $\Psi$ can be replaced by a finite set of coherent formulae.  Firstly, each geometric formula $\psi \in \Psi$ is, up to logical equivalence, a join of coherent formulae $\psi(\vec{x}) \equiv \bigvee_{\phi \in \Phi_\psi} \phi(\vec{x})$.  As $\theory$ proves the sequent $\top \vdash \bigvee_{\psi \in \Psi} \exists \vec{x}. \ \psi(\vec{x}) \equiv \bigvee_{\psi \in \Psi} \bigvee_{\phi \in \Phi_\psi} \exists \vec{x}. \ \phi(\vec{x})$, and since $\theory$ is coherent, by compactness there is a finite set of coherent formulae $\Psi' \subseteq \bigcup_{\psi \in \Psi} \Phi_\psi$ for which $\theory$ proves $\top \vdash \bigvee_{\phi \in \Psi'} \exists \vec{x}. \ \phi(\vec{x})$.  Next, for any model $M \vDash \theory$, an element $\vec{m} \in \Psi'(M)$ is also an element of $\Psi(M)$, since evidently $\Psi'(M) \subseteq \Psi(M)$.  Thus, we clearly see that, as $\theory$ is strongly uniformly rigid in sets over $\Psi$, it is also strongly uniformly rigid over $\Psi'$.  
		
		Viewed as an object of the coherent topos $\topos$, $\Psi'$ is a finite coproduct of \emph{coherent objects} (see \cite[\S D3.3]{elephant}), and hence coherent itself by \cite[Lemmas D3.3.3 \& D3.3.6]{elephant}.  Therefore, the slice $\topos/\Psi'$ is also a coherent topos by \cite[Lemma D3.3.16]{elephant}. The assumption that $\theory$ is strongly uniformly rigid in sets over $\Psi'$ precisely expresses that given any two pairs $(M \colon \Set \to \topos, m \colon 1 \to M^\ast \Psi')$ and $(N \colon \Set \to \topos, n \colon 1 \to N^\ast \Psi')$, equivalently described as a pair of models $M, N \vDash \theory$ and elements $m \in \Psi'(M)$ and $n \in \Psi'(N)$, there is at most one transformation $f \colon M^\ast \Rightarrow N^\ast$ making the diagram
		\[
		\begin{tikzcd}[row sep=tiny]
			& M^\ast \Psi \ar{dd}{f_{\Psi}} \\
			1 \ar[bend left = 20]{ru}{m} \ar[bend right = 20]{rd}[']{n} & \\
			& N^\ast \Psi
		\end{tikzcd}
		\] 
		commute, or in other words $\Topos(\Set,\topos/\Psi')$ is a preorder.  Thus, $\topos/\Psi'$ is a localic topos, completing the proof that $\topos$ is an \'etendue.
	\end{proof}
	Finally, we make the connection with Gaifman's co-ordinatisation theorem manifest in the classical setting via the following result, recovering the statement of the theorem given in the introduction.
	\begin{prop}\label{prop:classical_strong_rigid_is_rigid}
		If $\theory$ is a classical theory that is uniformly rigid in $\Set$, then $\theory$ is strongly uniformly rigid in $\Set$.
	\end{prop}
	\begin{proof}
		Recall that, for classical theories, the homomorphisms of models we consider are \emph{elementary embeddings} (cf.\ \cite[Lemma D1.5.13]{elephant}). Let $M, M' \vDash \theory$ be (set-based) models, let $\vec{m} \in \Psi(M)$ and $\vec{m}' \in \Psi(M')$, and suppose we have a pair of elementary embeddings $f, g \colon M \rightrightarrows M'$ that both send $\vec{m}$ pointwise to $\vec{m}'$.  We aim to show that $f = g$.
		
		We distinguish between two cases.  The first case, where $M$ is finite, is vacuous: note that any elementary embedding out of a finite structure is a bijection, and hence an isomorphism of models, since we can express the cardinality of the model $M$ in first-order logic.  Thus, we apply uniform rigidity to deduce that $f=g$.
		
		Now suppose that $M$ is infinite (we will actually observe, in \cref{coro:finite_models}, that this case becomes redundant).  Let $\kappa$ be a cardinal larger than $|M|$ and $|M'|$.  Using \cite[Theorem 10.2.1]{hodges}, we can elementarily embed $M$ and $M'$ into a strongly $\kappa$-homogeneous model $N \vDash \theory$.  By strong homogeneity, both $f$ and $g$ extend to automorphisms $\tilde{f}, \tilde{g} \colon N \rightrightarrows N$, as in the diagram
		\[
		\begin{tikzcd}
			\vec{m} \ar[maps to]{d} &[-30pt] \in &[-30pt] M \ar[shift right=2]{d}[']{f} \ar[shift left=2]{d}{g} \ar[hook]{r} & N \ar[shift right=2,dashed]{d}[']{\tilde{f}} \ar[shift left=2,dashed]{d}{\tilde{g}} \\
			\vec{m}' &[-30pt] \in &[-30pt] M' \ar[hook]{r} & N.
		\end{tikzcd}
		\]
		Now $\tilde{f}$ and $\tilde{g}$ are automorphisms of the model $N$ that both send $\vec{m} \in \Psi(M) \subseteq \Psi(N)$ pointwise to $\vec{m}' \in \Psi(M') \subseteq \Psi(N)$.  Thus, by the uniform rigidity of $\theory$, we have that $\tilde{f} = \tilde{g}$, and so $f = g$.
	\end{proof}
	\begin{coro}\label{coro:mainthm_classical}
		For a classical theory $\theory$ with classifying topos $\topos$, the following are equivalent:
			\begin{enumerate}[label = (\Alph*)]
			\item $\theory$ is uniformly co-ordinatised,
			\item $\theory$ is uniformly rigid in sets,
			\item $\topos$ is an \'etendue.
		\end{enumerate}
	\end{coro}
	\begin{coro}\label{coro:finite_models}
		If $\theory$ is a classical (respectively, coherent) theory that is (resp., strongly) uniformly rigid in sets, then every set-based model of $\theory$ is finite.
	\end{coro}
	\begin{proof}
		The statement follows from \cref{thm:mainthm_coherent} and \cref{rem:finite_models}.
	\end{proof}
	\section{Consequences}\label{sec:appl}
	We now discuss the interplay of our logical characterisation with some previously known facts regarding \'etendues.  In the first example application (\cref{prop:small_set_of_points}), the logical perspective allows someone unfamiliar with the techniques of topos theory to verify a fact concerning \'etendues (\cref{rem:small_points_via_logic}).  In the second example, we translate a result of topos theory (Rosenthal's theorem \cite{rosenthal2}) into a non-trivial result about geometric logic (\cref{coro:rosenthal_logical}).
	\subsection{\'Etendues have a small set of points}\label{sec:points}
	The set-based models of $\Vectff$ are all finite by design, but the essentially unique set-based $G$-torsor is infinite for $G$ infinite.  This highlights a key difference between geometric logic and (finitary) first-order logic: the upwards L\"owenheim-Skolem theorem for the latter implies that a theory with an infinite model has unboundedly many.  Consequently uniformly co-ordinatised theories with infinite models are somewhat unnatural from a classical model-theoretic point of view since a uniformly co-ordinatised theory only ever has a small set of set-based models, or equivalently an \'etendue always has a small set of points.   We first give the topos-theorist's proof of this well-known fact, before mentioning how it can be deduced logically.
	\begin{prop}\label{prop:small_set_of_points}
		Any \'etendue has a small set of points up to isomorphism.
	\end{prop}
	\begin{proof}
		Let $\topos$ be an \'etendue over an object $A \in \topos$.  Note that, since $\topos/A$ is localic, it has a small set of points (up to isomorphism).  Consequently, there is a small set of points of $\topos$ that factor through $\topos/A \to \topos$.  We show that any point $p \colon \Set \to \topos$ factors through $\topos/A \to \topos$.  Since surjective \'etale geometric morphisms are stable under pullback, we have that
		\[
		\begin{tikzcd}
			\Set/Y \ar{r} \ar{d} & \topos/A \ar{d} \\
			\Set \ar{r}{p} & \topos
		\end{tikzcd}
		\]
		is a pullback for some non-empty set $Y$, specifically $Y \cong p^\ast A$, and hence the morphism $\Set/Y \to \Set$ admits a section, induced by any section of the support $Y \twoheadrightarrow 1$.  Therefore, $p$ factors as $\Set \to \Set/Y \to \topos/A \to \topos$.
	\end{proof}
	\begin{rem}\label{rem:small_points_via_logic}
		We could have also deduced the statement of \cref{prop:small_set_of_points} from the uniform co-ordinatisability of any theory classified by an \'etendue.  If $\theory$ is uniformly co-ordinatised over $\Psi$, then every element of a set-based model $M \vDash \theory$ is the unique witness of a formula $\theta(\vec{m},y)$ with $\vec{m} \in \psi(M) \subseteq \Psi(M)$ and $\theta \in \Theta_\psi$.  Let $K$ be the maximal cardinality of the sets $\lrset{\Theta_\psi}{\psi \in \Psi}$ of formulae witnessing the uniform co-ordinatisation of $\theory$ over $\Psi$.  Then any set-based model of $\theory$ has at most $K$ elements, and so $\theory$ has a small set of models up to isomorphism.
	\end{rem}
	\subsection{Reducible expansions}\label{sec:rosenthal}
	First proved in \cite[Theorem 3.1]{rosenthal2}, and later simplified in \cite{rosenthal} using a result of Freyd \cite{freyd}, Rosenthal shows that every topos is a \emph{hyperconnected quotient} of an \'etendue.  As a final application, we give the logical translation of Rosenthal's result: every geometric theory admits a \emph{reducible, conservative expansion} to a uniformly co-ordinatisable theory.  We have already seen that \'etendues correspond to uniformly co-ordinatisable theories, and so it remains describe the equivalence between hyperconnected quotients and reducible, conservative expansions.  This latter observation is found in \cite[\S 7]{TST}.
	\begin{df} 
		Let $\theory$ be a (single-sorted) geometric theory over a signature $\Sigma$.  By an \emph{expansion} of $\theory$, we mean a second (single-sorted) geometric theory $\theory^+$ over a signature $\Sigma^+$ and a choice of:
		\begin{enumerate}
			\item a unary formula $\phi(x)$ in $\Sigma^+$,
			\item and, for each geometric formula $\psi(\vec{x})$ over the signature $\Sigma$, a formula $\psi^+(\vec{x})$ over $\Sigma^+$ for which $\theory^+$ proves the sequent $\psi^+(\vec{x}) \vdash \phi(x_1) \land \dots \land \phi(x_n)$,
		\end{enumerate} 
		such that, moreover, for each consequence $\psi(\vec{x}) \vdash \chi(\vec{x})$ of $\theory$, the theory $\theory^+$ proves the sequent $\psi^+(\vec{x}) \vdash \chi^+(\vec{x})$.  
	\end{df}
	Thus, if $\theory^+$ is an expansion of $\theory$ as above, then for each model $M \vDash \theory^+$, the definable subset $\phi(M) \subseteq M$ carries the structure of a $\theory$-model (hence the terminology \emph{expansion}).  In \cite[\S 2]{hodgespillay}, expansions are called \emph{relativised expansions}.
	\begin{df}
		Let $\theory^+$ be an expansion of $\theory$.
		\begin{enumerate}
			\item We say that the expansion is \emph{conservative} if, given geometric formulae $\psi(\vec{x}),\chi(\vec{x})$ over $\Sigma$, whenever $\theory^+$ proves the sequent $\psi^+(\vec{x}) \vdash \chi^+(\vec{x})$, then $\theory$ proves the sequent $\psi(\vec{x}) \vdash \chi(\vec{x})$.
			\item (\cite[\S 2]{hodgespillay}) We say that the expansion has the \emph{reduction property} (called the $\emptyset$-reduction property in \cite{reduction_prop}), or that the expansion is \emph{reducible}, if for every geometric formula $\xi(\vec{x})$ over $\Sigma^+$, there is a formula $\xi^-(\vec{x})$ over $\Sigma$ such that there is a provable equivalence of formulae
			\[
			\xi(\vec{x}) \land \phi(x_1) \land \dots \land \phi(x_n) \equiv (\xi^-)^+(\vec{x}),
			\]
			modulo the theory $\theory^+$.
		\end{enumerate}
	\end{df}
	\begin{rem}
		In the multi-sorted context discussed in \cite[\S 7]{TST}, Caramello's \emph{hyperconnected expansions} play the role of reducible, conservative expansions.  We have elected to follow the adjectives used by the model theory community (cf.\ \cite{hodgespillay,reduction_prop}).
	\end{rem}
	\begin{prop}[Theorem 7.1.5 \cite{TST}]
		Let $f \colon \ftopos \to \topos$ be a hyperconnected geometric morphism.  For any theory $\theory$ classified by $\topos$, there is a theory $\theory^+$ classified by $\ftopos$ that is a conservative, reducible expansion of $\theory$.
	\end{prop}
	\begin{proof}
		This is just a combination of \cite[Theorem 7.1.5]{TST} and \cite[Lemma D1.4.13]{elephant}.
	\end{proof}
	Thus, by combining our characterisation of the theories classified by an étendue with \cite[Theorem 3.1]{rosenthal2}, we obtain the following.
	\begin{coro}\label{coro:rosenthal_logical}
		Every geometric theory admits a reducible, conservative expansion to a uniformly co-ordinatisable theory.
	\end{coro}
	\begin{ex}
		We give an explicit description of a reducible, conservative expansion to a uniformly co-ordinatisable theory in the case for the theory $\Dtheory_\infty$ of \emph{infinite decidable objects}, i.e.\ the theory over the signature with one binary relation $\neq$ and, as axioms, the sequents
		\begin{align*}
			x = x' \land x \neq x' & \vdash \bot, \\
			\top & \vdash x = x' \lor x \neq x', \\
			\top & \vdash \exists y . \ (y \neq x_1) \land \dots \land (y \neq x_n) & \text{for all $n \in \N$.}
		\end{align*}
		Recall that $\Dtheory_\infty$ is classified by the \emph{Schanuel topos}, which can be described as the topos ${\bf B}\Omega(\N)$ of continuous actions on discrete sets by the topological group $\Omega(\N)$ of permutations of the naturals $\N$, viewed as an amorphic countable set, equipped with the \emph{pointwise convergence topology} (see, for instance, \cite[Example D3.4.10]{elephant}).  On the other hand, the topos $\Set^{\Omega(\N)}$ of all $\Omega(\N)$-actions on sets, continuous and non-continuous, classifies the geometric theory of $\Omega(\N)$-torsors (\cref{ex:torsors}).  The inclusion of continuous $\Omega(\N)$-actions into all actions is the inverse image of a geometric morphism $\Set^{\Omega(\N)} \to {\bf B}\Omega(\N)$ which is moreover hyperconnected, as is easily checked (see also \cite[p.~225]{elephant}).  As it stands, this geometric morphism does not yield the theory of $\Omega(\N)$-torsors as a reducible, conservative expansion of the theory of decidable objects.  We can instead replace the theory of $\Omega(\N)$-torsors with a \emph{Morita equivalent} theory, i.e.\ a geometric theory with the same classifying topos.

		Consider the following (single-sorted) geometric theory $\Dtheory_\infty^+$.
		\begin{enumerate}
			\item We have two unary relation symbols, $T(x)$ and $D(x)$, along with the axioms
			\[
			\top \vdash T(x) \lor D(x) , \quad T(x) \land D(x) \vdash \bot, \quad \top \vdash \exists x. \ T(x).
			\]
			\item For each $\sigma \in \Omega(\N)$, we add an endomorphism $\overline{\sigma}$ of $T(x)$ and the axioms
			\[
			T(x) \vdash\overline{e} x = x, \quad T(x)  \vdash \overline{\sigma} \, \overline{\sigma}' x = \overline{(\sigma \sigma')}x, \] \[
			T(x) \land T(y) \vdash \bigvee_{\sigma \in \Omega(\N)} y = \overline{\sigma}x, \ \text{ and } \ 
			\overline{\sigma} x = x  \vdash \bot \ \text{for all $\sigma \in \Omega(\N)\setminus e$.}
			\]
			Thus, $T(x)$ is endowed with the structure of an $\Omega(\N)$-torsor.
			\item We add a binary relation $\neq_D$ and the axioms
			\begin{align*}
				x \neq_D x' & \vdash D(x) \land D(x'), \\
				D(x) \land D(x') \land x = x' \land x \neq_D x' & \vdash \bot, \\
				D(x) \land D(x') & \vdash x = x' \lor x \neq_D x', \\
				D(x_1) \land \dots \land D(x_n) & \vdash \exists y . \ D(y) \land (y \neq_D x_1) \land \dots \land (y \neq_D x_n) 
			\end{align*}
			for all $n \in \N$, expressing that the equality predicate restricted to $D(x) \land D(x')$ is decidable and $D(x)$ is infinite.
			\item We add a function $q$ from $T(x)$ to $D(x)$ along with the axioms
			\begin{align*}
				D(x) & \vdash \exists y. \ T(y) \land q(y) = x, \\
				q(y) = q(y')  \dashv &\vdash T(y) \land T(y') \land \bigvee_{\sigma \in \stab(0)} y = \overline{\sigma} (y'), \\
				q(y) \neq_D q(y') \dashv & \vdash T(y) \land T(y') \land \bigvee_{\sigma \not \in \stab(0)} y = \overline{\sigma} (y'),
			\end{align*}
			where $\stab(0)$ denotes the stabilizer of the element $0 \in \N$ (since $\N$ is amorphic, we could have chosen any element).  In other words, $D(x)$ is the quotient of $T(x)$ by the definable equivalence relation $x \sim x' \equiv \bigvee_{\sigma \in \stab(0)}x = \overline{\sigma} (x')$, and moreover the interpretation of $\neq_D$ is fixed by this quotient.
		\end{enumerate}
		As the interpretation of $D(x)$ and $\neq_D$ in any model is fixed by the interpretation of the $\Omega(\N)$-torsor $T(x)$, the theory $\Dtheory_\infty^+$ is Morita equivalent to the theory of $\Omega(\N)$-torsors, i.e.\ the classifying topos of $\Dtheory_\infty^+$ is given by $\Set^{\Omega(\N)}$.  Indeed, the \emph{pre-bound} of $\Set^{\Omega(\N)}$ corresponding to the single-sorted theory $\Dtheory_\infty^+$ (under the bijection in \cite[Theorem D3.2.5]{elephant}) is given by the object $\Omega(\N) + \Omega(\N)/\stab(0)$, equipped with the obvious $\Omega(\N)$-action (whereas the pre-bound corresponding to the theory of $\Omega(\N)$-torsors is just $\Omega(\N)$ equipped with the canonical action).
		
		Thus, $\Dtheory_\infty^+$ is uniformly co-ordinatised, as can be seen explicitly: the theory is uniformly co-ordinatised over the inhabited formula $T(x)$, with the formulae witnessing the uniform co-ordinatisation given by
		\[
		\sett{T(x) \land T(y) \land (y = \overline{\sigma}(x))}{\sigma \in \Omega(\N)}
		\cup \sett{T(x) \land D(y) \land (y = q (\overline{\sigma}(x)))}{\sigma \in \Omega(\N)}.
		\]
		The theory $\Dtheory_\infty$ admits an evident expansion to $\Dtheory_\infty^+$, given by the choice of the unary formula $D(x)$ and taking $(x \neq x')^+ $ to be $x \neq_D x'$.  This expansion is reducible and conservative.  This can be seen topos-theoretically via \cite[Theorem 7.1.3]{TST}, as the geometric morphism corresponding to the expansion (in the sense of \cite[\S 7.1]{TST}) is the aforementioned hyperconnected morphism $\Set^{\Omega(\N)} \to {\bf B}\Omega(\N)$.  In more detail, the pre-bound $\Omega(\N) / \stab(0)$ of ${\bf B}\Omega(\N)$, corresponding to the theory $\Dtheory_\infty$ classified by ${\bf B}\Omega(\N)$, is sent by the inverse image functor to the subobject $\Omega(\N)/\stab(0) \subseteq \Omega(\N) + \Omega(\N)/\stab(0)$, which witnesses the expansion of $\Dtheory_\infty$ into $\Dtheory_\infty^+$.
	\end{ex}

	\subsection*{Acknowledgements}
	I am grateful to Ryuya Hora who initially asked me the question ``what theories do \'etendues classify''? Jonas Frey, Morgan Rogers, and the anonymous reviewer deserve mention for pointing out mistakes in earlier drafts.  Finally, I thank Sam van Gool, Ivan Toma\v{s}i\'c, Behrang Noohi and Rui Prezado for suggesting improvements to the paper.  This work was supported by Agence Nationale de la Recherche (ANR), project ANR-23-CE48-0012-01.
	
	\smallskip
	
	\epigraph{\it 
		J'\'etais au bord du vide, comme dans un r\^eve.  Il y a l'\'etendue d'ombre, on peut tomber.}{Le Cl\'ezio}
	\printbibliography

@book{hodges,
	author = {Wilfred Hodges},
	title = {Model theory},
	publisher = {Cambridge University Press},
	year = {1993}
}

@article{relcat,
	title = {Omega-categoricity, relative categoricity and coordinatisation},
	journal = {Annals of Pure and Applied Logic},
	volume = {46},
	number = {2},
	pages = {169-199},
	year = {1990},
	author = {Wilfrid Hodges and I.M. Hodkinson and Dugald Macpherson}
}

@article{JT,
	title={An extension of the {G}alois theory of {G}rothendieck},
	author={Joyal, Andr\'{e} and Tierney, Myles},
	journal={Memoirs of the American Mathematical Society},
	volume={51},
	number={309},
	year={1984}
}

@book{SGL,
	title = {Sheaves in Geometry and Logic},
	subtitle = {A First Introduction to Topos Theory},
	author = {Mac Lane, Saunders and Moerdijk, Ieke},
	publisher = {Springer, New York},
	year = {1994}
}

@book{elephant,
	title = {Sketches of an {E}lephant: a Topos Theory Compendium, vols. 1 and 2},
	author = {Peter T. Johnstone},
	publisher = {Oxford University Press},
	year = {2002},
}

@book{topos,
	author = {Peter T. Johnstone},
	title = {Topos theory},
	publisher = {Academic Press, New York},
	year = {1977}
}

@book{MR,
	title={First Order Categorical Logic: Model-Theoretical Methods in the Theory of Topoi and Related Categories},
	author={Michael Makkai and Gonzalo E. Reyes},
	series={Lecture Notes in Mathematics},
	year={1977},
	publisher={Springer Berlin, Heidelberg}
}

@book{SGA4-3,
	title={Th\'{e}orie des Topos et Cohomologie Etale des Sch\'{e}mas},
	titleaddon={S\'{e}minaire de G\'{e}om\'{e}trie Alg\'{e}brique du Bois-Marie 1963-1964 (SGA 4)},
	author={Artin, M. and Grothendieck, A. and Verdier, J.L.},
	series={Lecture Notes in Mathematics},
	year={1973},
	number={305},
	publisher={Springer Berlin, Heidelberg}
}

@article{rosenthal,
	author = {Kimmo I. Rosenthal},
	title = {Covering \'etendues and {F}reyd's theorem},
	journal = {Proceedings of the American Mathematical Society},
	year = {1987},
	volume = {99},
	number = {2},
	pages = {221-222}
}

@article{freyd,
	title={All topoi are localic or why permutation models prevail},
	author={Peter Freyd},
	journal={Journal of Pure and Applied Algebra},
	volume={46},
	pages={49--58},
	year={1987}
}

@article{hodgespillay,
	author = {Hodges, Wilfrid and Pillay, Anand},
	title = {Cohomology of Structures and Some Problems of {A}hlbrandt and {Z}iegler},
	journal = {Journal of the London Mathematical Society},
	volume = {50},
	number = {1},
	pages = {1-16},
	year = {1994},
}

@article{reduction_prop,
	author = {Hirotaka Kikyo and Akito Tsuboi},
	journal = {The Journal of Symbolic Logic},
	number = {3},
	pages = {900--911},
	title = {On Reduction Properties},
	volume = {59},
	year = {1994}
}

@article{fact1,
	author = {Johnstone, P. T.},
	title = {Factorization theorems for geometric morphisms, {I}},
	journal = {Cahiers de Topologie et G\'eom\'etrie Diff\'erentielle},
	pages = {3--17},
	volume = {22},
	number = {1},
	year = {1981},
}

@unpublished{jens,
	author = {Jens Hemelaer},
	title = {Geometric morphisms to slice toposes},
	url = {https://jenshemelaer.com/slice-toposes.pdf},
	year = {2021}
}

@unpublished{resende,
	author = {Pedro Resende},
	title = {Lectures on étale groupoids, inverse
	semigroups and quantales},
	url = {https://www.math.tecnico.ulisboa.pt/~pmr/poci55958/gncg51gamap-version2.pdf},
	year = {2006}
}

@article{evanshewitt,
	title = {Counterexamples to a conjecture on relative categoricity},
	journal = {Annals of Pure and Applied Logic},
	volume = {46},
	number = {2},
	pages = {201-209},
	year = {1990},
	author = {David M. Evans and P.R. Hewitt}
}

@article{caramello_atomic,
	author = {Olivia Caramello},
	title = {Atomic toposes and countable categoricity},
	journal = {Applied Categorical Structures},
	year = {2011},
	volume = {20},
	pages = {379--391}
}

@book{TST,
	publisher = {Oxford University Press},
	title = {Theories, sites, toposes: relating and studying mathematical theories through topos-theoretic `bridges'},
	year = {2018},
	author = {Caramello, Olivia}
}

@article{local_eqv,
	title = {Every étendue comes from a local equivalence relation},
	journal = {Journal of Pure and Applied Algebra},
	volume = {82},
	number = {2},
	pages = {155-174},
	year = {1992},
	author = {Anders Kock and Ieke Moerdijk},
}

@article{mono_site,
	author = {Kock, Anders and Moerdijk, Ieke},
	title = {Presentations of \'etendues},
	journal = {Cahiers de Topologie et G\'eom\'etrie Diff\'erentielle Cat\'egoriques},
	pages = {145--164},
	volume = {32},
	number = {2},
	year = {1991},
}

@article{moer_zara,
	author = {Ieke Moerdijk},
	title = {Foliations, groupoids and {G}rothendieck \'etendues},
	journal = {Revista de la Academia de Ciencias de Zaragoza},
	volume = {48},
	year = {1993},
	pages = {5--33}
}

@article{moer_fourier,
	author = {Moerdijk, Ieke},
	title = {Classifying toposes and foliations},
	journal = {Annales de l'Institut Fourier},
	pages = {189--209},
	volume = {41},
	number = {1},
	year = {1991},
}

@article{rosenthal_local_eqv,
	author = {Rosenthal, Kimmo I.},
	title = {Sheaves and local equivalence relations},
	journal = {Cahiers de Topologie et G\'eom\'etrie Diff\'erentielle Cat\'egoriques},
	pages = {179--206},
	volume = {25},
	number = {2},
	year = {1984},
}

@article{moerpronk_orbifolds,
	author = {Ieke Moerdijk and Dorette A. Pronk},
	title = {Orbifolds, sheaves and groupoids},
	journal = {K-Theory},
	volume = {12},
	pages = {3--21},
	year = {1997}
}

@article{rosenthal2,
	author = {Rosenthal, Kimmo I.},
	title = {Quotient systems in {Grothendieck} topoi},
	journal = {Cahiers de Topologie et G\'eom\'etrie Diff\'erentielle},
	pages = {425--438},
	volume = {23},
	number = {4},
	year = {1982},
}

@article{rosenthal_monic,
	title = {Étendues and categories with monic maps},
	journal = {Journal of Pure and Applied Algebra},
	volume = {22},
	number = {2},
	pages = {193-212},
	year = {1981},
	author = {Kimmo I. Rosenthal}
}

@article{lawson_boolean_monoids,
	author = {M. V. Lawson},
	title = {A noncommutative generalization of {S}tone duality},
	journal = {Journal of the Australian Mathematical Society},
	pages = {385--404},
	volume = {88},
	year = {2010}
}

@article{cockettgarner,
	title = {Generalising the étale groupoid–complete pseudogroup correspondence},
	journal = {Advances in Mathematics},
	volume = {392},
	pages = {108030},
	year = {2021},
	author = {Robin Cockett and Richard Garner},
}

@article{lawson_semigroups,
	author = {Lawson, M. V.},
	title = {Non-commutative {S}tone duality: inverse semigroups, topological groupoids and $C^*$-algebras},
	journal = {International Journal of Algebra and Computation},
	volume = {22},
	number = {06},
	pages = {1250058},
	year = {2012},
}

@book{marker,
	author = {David Marker},
	title = {Model Theory: An Introduction},
	publisher = {Springer New York},
	year = {2002},
	series = {Graduate Texts in Mathematics}
}

@article{milnor,
	author = {J. Milnor},
	title = {Construction of universal bundles, {II}},
	journal = {Annals of Mathematics},
	volume = {63},
	number = {3},
	pages = {430--436},
	year = {1956}
}

@incollection{lawvere_etendue,
	title = {Variable Quantities and Variable Structures in Topoi},
	editor = {Alex Heller and Myles Tierney},
	booktitle = {Algebra, Topology, and Category Theory},
	publisher = {Academic Press},
	pages = {101-131},
	year = {1976},
	author = {F. W. Lawvere},
}

@article{hofstra_funk_invitation,
	title = {Toposes for semigroups: an invitation},
	journal = {Semigroup Forum},
	volume = {103},
	pages = {715--776},
	year = {2021},
	author ={J. Funk and P. Hofstra}
}

@article{myselfgrpds,
	author = {Wrigley, Joshua L.},
	title = {On topological groupoids that represent theories},
	journal = {Zeitschrift für Mathematische Logik und Grundlagen der Mathematik},
	volume = {to appear},
	year = {2026},
	pages = {1--44},
}
	
\end{document}